\documentclass[10pt,letterpaper]{article}
\makeatletter
\usepackage{mathrsfs}
\usepackage{amscd}
\usepackage{amsmath}
\usepackage{latexsym}
\usepackage{amsfonts}
\usepackage{amssymb}
\usepackage{amsthm}
\usepackage{graphicx}
\usepackage{enumerate}

\newtheorem{theorem}{\textbf Theorem}[section]
\newtheorem{lemma}{\textbf Lemma}[section]

\newcommand{\be}{\begin{eqnarray}}
\newcommand{\ee}{\end{eqnarray}}

\newcommand{\bes}{\begin{eqnarray*}}
\newcommand{\ees}{\end{eqnarray*}}

\begin{document}
\begin{titlepage}
\title{\bf  Stability and inviscid limit of the 3D anisotropic MHD system near a background magnetic field with mixed fractional partial dissipation}
\author{
        Xuemin Deng, Yuelong Xiao
        \\School of Mathematics and Computational Science,
       \\ Xiangtan University,  Xiangtan, Hunan,  411100,  China.
       \\ Aibin Zang\thanks{Corresponding Author: A. Zang}
       \\ The Center of Applied Mathematics, \\Yichun University Yichun, Jiangxi, 336000, China.
       \\( dxm@smail.xtu.edu.cn, xyl@xtu.edu.cn, abzang@jxycu.edu.cn)
          }
\date{}
\end{titlepage}
\maketitle

\begin{abstract}
A main result of this paper establishes the global stability of the three-dimensional MHD equations near a background magnetic field with mixed fractional partial dissipation with $\alpha, \beta\in(\frac{1}{2}, 1]$.
Namely, the velocity equations involve dissipation $(\Lambda_1^{2\alpha} + \Lambda_2^{2\alpha}+\sigma \Lambda_3^{2\alpha})u$ with the case $\sigma=1$ and $\sigma=0$.  The magnetic equations without partial magnetic diffusion $\Lambda_i^{2\beta} b_i $ but with the diffusion $(-\Delta)^\beta b$, where $\Lambda_i^{s} (s>0)$ with $i=1, 2, 3$ are the directional fractional operators.
Then we focus on the vanishing vertical kinematic viscosity coefficient limit of the MHD system with the case $\sigma=1$ to the case $\sigma=0$. The convergent result is obtained in the sense of $H^1$-norm.
\end{abstract}
\vspace{.2in} {\bf Key words:} MHD equations; Stability; Fractional partial dissipation; Inviscid limit.\\
\vspace{.2in} {\bf MSC:} 35B35;  35Q35;   76B03.
\section{Introduction}

~~~~The magnetohydrodynamic (MHD) system widely describes the dynamics of electrically conducting fluids such as plasmas, liquid metals, and salt water or electrolytes(see, e.g., \cite{BiskampDieter1993Nonlinear,Pippard1989inMetals}).
The 3D full fractional MHD system reads
\begin{eqnarray}\label{B}
\begin{cases}
\partial_{t} U+ U \cdot\nabla U =-\left(\nu_{1}\Lambda_1^{2\alpha} + \nu_{2}\Lambda_2^{2\alpha}+ \nu_{3}\Lambda_3^{2\alpha} \right) U+\nabla P-B \cdot\nabla B,  \hspace{0.5cm} (t, x)\in \mathbb{R}_{+}\times\mathbb{R}^3,\\
\partial_{t} B+ U \cdot\nabla B =-\mu\left(-\Delta\right)^\beta B-B \cdot\nabla U, \\
div ~U=div ~B=0, \\
(U, B)\vert _{t=0}=(U_{in}, B_{in}),
\end{cases}
\end{eqnarray}
where $U=\left(U_1, U_2, U_3\right)(t, x)$, $B=\left(B_1, B_2, B_3\right)(t, x)$ and $P=P(t, x)$ denote the three-component velocity field, the three-component magnetic field and the pressure, respectively.
The kinematic viscosity $\nu_1, \nu_2, \nu_3>0$ and the magnetic diffusivity $\mu>0$. The parameters $\alpha, \beta\geq0$ are fractional derivatives and the operator $(-\Delta)^{s}$ is fractional Laplacian and defined by Fourier transform,
$$\widehat{(-\Delta)^s f}(\xi)=|\xi|^{2s}\widehat{f}(\xi),$$
with
$$\widehat{f}(\xi)=\frac{1}{(2\pi)^{\frac{d}{2}}}\int_{\mathbb{R}^d} e^{-ix\cdot\xi}f(x)dx.$$
 $\Lambda_i^{s}$ with $i=1, 2, 3$ is the directional fractional operator and defined by Fourier transform,
$$\widehat{\Lambda_i^s f}(\xi)=|\xi_i|^{2s}\widehat{f}(\xi), \quad \xi=(\xi_1, \xi_2, \xi_3).$$

For the standard MHD equations $(\alpha=\beta=1)$ with full dissipation or partial dissipation, the global well-posedness
problem has recently attracted considerable attention and significant progress has been made
(see~\cite{CaoChongsheng2011horizontal, CaoChongsheng2011regularity, CaoChongsheng2014onlymagnetic, SermangeMichel1983MHD, WangFan2013mixedpartial, WuJiahong2003MHD}).
For the MHD system with full fractional dissipation or no dissipation, some results can be referenced (see~\cite{DaiYichen20203fractional, DaiYichen20192fractional, FanJishan20142fractional, JiEunjeong20132fractional, JiuQuansen20203fractional, WangWeihua20193fractional, WuJiahong20182fractional}).

In recently years, there have also been some results of the global well-posedness of the MHD equations involved only partial fractional dissipation. Dong et al. \cite{DongBQ2018partialmagnetic} obtained the 2D MHD system with the full fractional velocity dissipation and partial magnetic diffusion has a unique global solution in the initial data $(u_{in}, b_{in})\in H^s(\mathbb{R}^2), s\geq3$.
Then, for the 2D MHD equations ($\beta>1$) with the only partial magnetic diffusion and no velocity dissipation, Dong, Jia and Wu \cite{DongBQ2019onlymagnetic} proved that it has a globally unique solution in the initial data $(u_{in}, b_{in})\in H^s(\mathbb{R}^2), s>2$.
In \cite{yang2019fractional}, Yang, Wu and Jiu established the global existence and uniqueness of strong solutions to the 3D hyperdissipative MHD equations with only partial velocity dissipation and horizontal magnetic diffusion.
Our aim is to understand the stability problem of perturbations near a background magnetic field of the following MHD system both the full dissipation case and mix fractional partial dissipation case as,
\begin{eqnarray}\label{BHE}
\begin{cases}
\partial_{t} U+ U \cdot\nabla U=-\left(\nu_{1}\Lambda_1^{2\alpha} + \nu_{2}\Lambda_2^{2\alpha}+\sigma \nu_{3}\Lambda_3^{2\alpha} \right)U+\nabla P-B \cdot\nabla B, \hspace{0.5cm} (t, x)\in \mathbb{R}_{+}\times\mathbb{R}^3,\\
\partial_{t} B+ U \cdot\nabla B=-\mu\begin{bmatrix} \Lambda_2^{2\beta} + \Lambda_3^{2\beta}\\ \Lambda_1^{2\beta} + \Lambda_3^{2\beta}\\ \Lambda_1^{2\beta} + \Lambda_2^{2\beta} \end{bmatrix} B-B \cdot\nabla U, \\
div ~U=div ~B=0, \\
(U, B)\vert _{t=0}=(U_{in}, B_{in}),
\end{cases}
\end{eqnarray}
where $\sigma=1$ corresponds to the full dissipation case and $\sigma=0$ to the horizontal dissipation case.
For the $\sigma=0$ case, the each velocity equation has dissipation in the same direction and the each magnetic equation does not magnetic diffusion in the same direction.
Namely, for the velocity equations, the each component equation involves only horizontal dissipation.
For the magnetic equations of the system (\ref{BHE}), the first component equation involves no fractional dissipation in the $x_1$-direction, the second component equation involves no fractional dissipation in the $x_2$-direction and the third component equation involve no fractional dissipation in the $x_3$-direction.

The equilibrium
$$ (U_e, B_e, P_e)=(0, e_3, 0)$$
is special steady solution of the system (\ref{BHE}). The perturbations near a background magnetic field
$$u=U-U_e, \quad b=B-B_e, \quad p=P-P_e$$
satisfy
\begin{eqnarray}\label{bhe}
\begin{cases}
\partial_{t} u+ u \cdot\nabla u=-\left(\nu_{1}\Lambda_1^{2\alpha} + \nu_{2}\Lambda_2^{2\alpha}+ \sigma\nu_{3}\Lambda_3^{2\alpha} \right) u+\nabla p-b \cdot\nabla b+\partial_3 b, ~(t, x)\in \mathbb{R}_{+}\times\mathbb{R}^3,\\
\partial_{t} b+ u \cdot\nabla b~=-\mu\begin{bmatrix} \Lambda_2^{2\beta} + \Lambda_3^{2\beta}\\ \Lambda_1^{2\beta} + \Lambda_3^{2\beta}\\ \Lambda_1^{2\beta} + \Lambda_2^{2\beta} \end{bmatrix} b-b \cdot\nabla u+\partial_3 u, \\
div~u=div~b=0, \\
(u, b)\vert _{t=0}=(u_{in}, b_{in}) . \\
\end{cases}
\end{eqnarray}
The full fractional Laplacian dissipation can be replaced by a partial fractional Laplacian dissipation of the fractional MHD system in certain physical regimes and under suitable scaling. In addition, the global stability of the MHD equations near a background magnetic field which has recently drawn a lot of attention.

In two dimensions, Hu and Lin \cite{Lin2014tability} obtained the stability of the MHD equations with with zero magnetic diffusivity. Lin et al. \cite{Lin2020tability} researched the stabilization effect of a background magnetic field of the standard MHD system with only vertical velocity dissipation and horizontal magnetic diffusion.
The stability of the MHD equations was established in which only the vertical component equation in the velocity equation and the horizontal component equation in the magnetic equation involve full dissipation by Li, Wu and Xu \cite{XuXiaojing2020tability}.
Feng, Wang and Wu \cite{Feng2023tability} investigated the stability of the MHD equations with only horizontal fractional velocity dissipation and magnetic diffusion dissipation. More results on the stability of various partially or fractionally dissipated MHD equations under the magnetic background can be found in these references \cite{Boardman2020stability, Lai2021tability, Lai2022tability, LinFanghua2015mhd, GuoYana2022tability}.

In three dimensions, the stability problem of the standard MHD system near a background magnetic field was studied by Wu and Zhu \cite{Wujiahong2021tability} in the case when the velocity dissipation only occurs in the horizontal direction and the magnetic diffusion only occurs in the vertical direction.
Ji and Tian \cite{JiRuihong2021tability} showed the stability of the standard MHD equations with only horizontal velocity dissipation and magnetic diffusion in periodic domain.
Very recently, it was proved by Ji, Jiang and Luo \cite{JiRuihong2023tability} and independently by Li, Wang and Zheng \cite{LiJingna2023tability} that the stability of the MHD equations with only fractional horizontal velocity dissipation and magnetic diffusion in the different fractional range $\alpha, \beta\in(\frac{1}{2}, 1]$ and $\alpha=\beta\in[\frac{1}{2}, 1)$, respectively.
More results on the stability of the MHD system can be referred to (\cite{AbidiHammadi2017tability, DengWen2018tability, LinFanghua2014mhd, ShangHaifeng2022tability, LuCheng2022tability, ZhengDahao2023tability}).

The MHD equations with fractional partial dissipation is very important in mathematically.
Another purpose of our study is to understand the stability of a family of the MHD system (\ref{bhe}) when the sizes of the index $\alpha$ and $\beta$ of the partial fractional dissipation vary.
In contrast to \cite{JiRuihong2023tability},
the each component of the magnetic equations of the MHD system we investigated does not involve its own dissipation.
More precisely, due to the lack of the dissipation $\Lambda_1^{2\beta} b_1$, $\Lambda_2^{2\beta} b_2$ and $\Lambda_3^{2\beta} b_3$ in the magnetic equations of the (\ref{bhe}), we have to overcome some new difficulties.

Here we present two main results in this paper. The first is to investigate the global stability of the MHD system (\ref{bhe}) near a background magnetic field with $\alpha, \beta\in(\frac{1}{2}, 1]$ in $H^3(\mathbb{R}^3)$.

Our main result reads:
\begin{theorem}\label{theoreml}
Consider the system $(\ref{bhe})$ for $\alpha, \beta\in(\frac{1}{2}, 1]$. Let the initial data $(u_{in}, b_{in})\in H^3(\mathbb{R}^3)$ and $div~u_{in}=div~b_{in}=0$ such that
$$\|(u_{in}, b_{in})\|_{H^3}\leq\epsilon$$
for some sufficiently small $\varepsilon>0$.
Then there exists a unique global solution $(u, \theta)$ satisfying
\begin{multline}\label{neng}
\|(u, b)\|_{H^3}^2+\int_0^t \|(\nu_1^{\frac{1}{2}}\Lambda_1^{\alpha}, \nu_2^{\frac{1}{2}}\Lambda_2^{\alpha})u\|_{H^3}^2 d\tau+\sigma\nu_3\int_0^t \|\Lambda_3^{\alpha}u\|_{H^3}^2 d\tau \\
+\mu\int_0^t \|(\Lambda_2^{\beta}, \Lambda_3^{\beta})b_1\|_{H^3}^2+\|(\Lambda_1^{\beta}, \Lambda_3^{\beta})b_2\|_{H^3}^2+\|(\Lambda_1^{\beta}, \Lambda_2^{\beta})b_3\|_{H^3}^2 d\tau\leq C\epsilon^2,
\end{multline}
where the constant $C>0$ is independent of $\epsilon$ and $t$.
\end{theorem}
Besides the study on the stability of the MHD equations, many scholars has been made many excellent work on the inviscid limit of the standard MHD system with full dissipation (see, e. g. , \cite{DuanQin2022Vanishing, FanJishan2007Vanishing, MaafaYoussouf2023Vanishing, XiaoYuelong2009Vanishing, WangNa2018Vanishing, WuZhonglin2018Vanishing}).
For the 3D standard anisotropic MHD equations with Dirichlet boundary conditions, Wang and Wang \cite{WangShu2017Vanishing} provided a $L^2$-convergence result of the viscous solutions when the vertical kinematic viscosity coefficient and magnetic diffusion coefficient go to zero. For the fractional MHD equations with full or partial dissipation, the vanishing viscosity limit problem has been researched relatively little.

Our second goal is to understand its inviscid limit problem which the following MHD system as the vertical viscosity coefficient $\nu$ goes to zero.
\begin{eqnarray}\label{bhev}
\begin{cases}
\partial_{t} u^{\nu}+ u^{\nu} \cdot\nabla u^{\nu}=-~\left(\Lambda_1^{2\alpha} + \Lambda_2^{2\alpha}+ \nu\Lambda_3^{2\alpha} \right) ~u^{\nu}+\nabla p^{\nu}-b^{\nu} \cdot\nabla b^{\nu}+\partial_3 b^{\nu}, \\
\partial_{t} b^{\nu}+ u^{\nu} \cdot\nabla b^{\nu}~=-\begin{bmatrix} \Lambda_2^{2\beta} + \Lambda_3^{2\beta}\\ \Lambda_1^{2\beta} + \Lambda_3^{2\beta}\\ \Lambda_1^{2\beta} + \Lambda_2^{2\beta} \end{bmatrix} b^{\nu}-b^{\nu} \cdot\nabla u^{\nu}+\partial_3 u^{\nu}, \\
div~u^{\nu}=div~b^{\nu}=0, \\
(u^{\nu}, b^{\nu})\vert _{t=0}=(u_{in}, b_{in}). \\
\end{cases}
\end{eqnarray}
When the vertical viscosity coefficient $\nu$ is vanishing, the MHD system (\ref{bhev}) degenerates into the system (\ref{bhev0}) as follow
\begin{eqnarray}\label{bhev0}
\begin{cases}
\partial_{t} u^{0}+ u^{0} \cdot\nabla u^{0}=-~\left(\Lambda_1^{2\alpha} + \Lambda_2^{2\alpha} \right)~ u^{0}+\nabla p^{0}-b^{0} \cdot\nabla b^{0}+\partial_3 b^{0}, \\
\partial_{t} b^{0}+ u^{0} \cdot\nabla b^{0}~=-\begin{bmatrix} \Lambda_2^{2\beta} + \Lambda_3^{2\beta}\\ \Lambda_1^{2\beta} + \Lambda_3^{2\beta}\\ \Lambda_1^{2\beta} + \Lambda_2^{2\beta} \end{bmatrix} b^{0}-b^{0} \cdot\nabla u^{0}+\partial_3 u^{0}, \\
div~u^{0}=div~b^{0}=0, \\
(u^{0}, b^{0})\vert _{t=0}=(u_{in}, b_{in}). \\
\end{cases}
\end{eqnarray}

Under the assumptions of Theorem \ref{theoreml}, the systems (\ref{bhev}) and (\ref{bhev0}) have a unique global respectively solution $(u^{\nu}, b^{\nu})$ and $(u^{0}, b^{0})$ satisfying (\ref{neng}).

Our second main result presents the inviscid limit problem.
\begin{theorem}\label{theorem2}
Assume that $(u^{\nu}, b^{\nu})$ and $(u^{0}, b^{0})$ are respectively the solution to the system (\ref{bhev}) and the system (\ref{bhev0}) with the same initial data $(u_{0}, b_{0})$ under the conditions in Theorem \ref{theoreml}.
Then we have
\begin{align*}
\|u^{\nu}-u^{0}\|_{{L^{\infty}_t}H^1}+\|b^{\nu}-b_{0}\|_{{L^{\infty}_t}H^1}\leq C(t)\nu.
\end{align*}
\end{theorem}
The detailed statements of Theorem \ref{theoreml} and \ref{theorem2} can be found in Sections~\ref{s3} and \ref{sec4}. We now state a summary of the major obstacles encountered and the main proof ideas.
The proof of Theorem \ref{theoreml} is divided into twofold. First step, we obtain the global stability by using standard energy method. And we structure the energy functional
\begin{align}\label{energyfunctional}
E(t)=&\underset{0\leq\tau\leq t}{sup}\|(u, b)\|_{H^3}^2 +\int_0^t \|(\nu_1^{\frac{1}{2}}\Lambda_1^{\alpha}, \nu_2^{\frac{1}{2}}\Lambda_2^{\alpha})u\|_{H^3}^2 d\tau+\sigma\nu_3\int_0^t \|\Lambda_3^{\alpha}u\|_{H^3}^2 d\tau \nonumber\\
&+\mu\int_0^t\|(\Lambda_2^{\beta}, \Lambda_3^{\beta})b_1\|_{H^3}^2+\|(\Lambda_1^{\beta}, \Lambda_3^{\beta})b_2\|_{H^3}^2+\|(\Lambda_1^{\beta}, \Lambda_2^{\beta})b_3\|_{H^3}^2 d\tau
\end{align}
to satisfy
\begin{align}\label{et}
E(t)\leq E(0)+CE(t)^{\frac{3}{2}}, \quad \forall t\geq 0.
\end{align}
The second step is to obtain the stability (\ref{neng}) by using the bootstrapping argument. Details of this approach refer to Section \ref{s3} of this paper.
However, the process of overcoming the stability problem is not extremely trivial. On the one hand, we first recall that the system(\ref{bhe}) has only partial fractional dissipation, namely, about the velocity equations involve only horizontal dissipation and the lack of dissipation in its own direction in the each magnetic component equation. Due to the lack of full dissipation, it becomes difficult for us to control nonlinear terms.
Thanks to the divergence free condition and the general anisotropic inequality of the Lemma \ref{lemmage} (iii, iv) help complete the bounds. On the other hand, the immediate index of the triple product norms of the general anisotropic inequality is invalid. We want to configure the desired index by using the interpolation inequality of Lemma \ref{lemmage} (i, ii) repeatedly. Eventually, we obtain the bound in the proof of Theorem \ref{theoreml}.

The rest of the paper is organized as follows. In Section~\ref{sec:pre}, we introduce some notations and give definitions of some function spaces. We mainly present some useful lemmas in this section. In Section ~\ref{s3}, we show the proof of theorem \ref{theoreml}: we prove the stability of the system (\ref{bhe}) via the energy method and bootstrapping argument and then we get the uniqueness part. Finally, the vanishing vertical kinematic viscosity limit for the system (\ref{bhev}) is discussed in Section ~\ref{sec4}.

\section{Preliminaries}\label{sec:pre}
Throughout this paper for the convenience of writing, C stands for some real positive constant. Denote $$\|f\|^2:=\|f\|^2_{L^2(\mathbb{R}^3)}, ~\|f\|_{L_t^2 H^s}:=\int_0^t\|f\|^2_{H^s(\mathbb{R}^3)} d\tau, ~s\geq0 $$
and
\begin{align*}
\|(\Lambda^{s_1}, \Lambda^{s_2}) (f,g)\|^2 &:= \|\Lambda^{s_1} (f,g)\|^2+\|\Lambda^{s_2} (f,g)\|^2\\&:= \|\Lambda^{s_1} f\|^2+ \|\Lambda^{s_1} g\|^2+\|\Lambda^{s_2} f\|^2+\|\Lambda^{s_2} g\|^2,  ~s_1,s_2\geq0,
\end{align*}
where  $f=f(t)$ and $ g=g(t)$ are two measurable functions on $\mathbb{R}^3$. 
The inhomogeneous anisotropic Sobolev space and the homogeneous anisotropic space, denoted by $H^s(\mathbb{R}^3)$ and $\dot{H}^s(\mathbb{R}^3)$, are respectively defined as follows
\begin{align*}
H^s(\mathbb{R}^3)&:=\left\{f\in H^s(\mathbb{R}^3)\bigg|\int_{\mathbb{R}^3} \prod_{i=1}^3(1+|\xi_{i}|^2)^{s_i}|\hat{f}(\xi)|^2 d\xi<\infty \right\}\\
\dot{H}^s(\mathbb{R}^3)&:=\left\{f\in \dot{H}^s(\mathbb{R}^3)\bigg|\int_{\mathbb{R}^3} \prod_{i=1}^3|\xi|_{i}^{2s_i}|\hat{f}(\xi)|^2 d\xi<\infty \right\},
\end{align*}
where $ \xi=(\xi_1, \xi_2, \xi_3)$, $s=(s_1, s_2, s_3)$ and $|s|=s_1+s_2+s_3$.

Now to introduce some essential inequalities for the anisotropic Sobolev space.
\begin{lemma}(see \cite{Jiu2019chazhi}.)\label{lemmachazhi}
Let $2\leq q \leq\infty $ and $s>d(\frac{1}{2}-\frac{1}{q})$ such that the following estimates hold true
$$ \| f\|_{L^q(\mathbb{R}^d)}\leq C \| f\|_{L^2(\mathbb{R}^d)}^{1-\frac{d}{s}(\frac{1}{2}-\frac{1}{q})} \|\Lambda^s f\|_{L^2(\mathbb{R}^d)}^{\frac{d}{s}(\frac{1}{2}-\frac{1}{q})}, \quad \forall f\in H^s(\mathbb{R}^d),$$
where the constant $C$ depends on $d, p$ and $s$.
\end{lemma}

\begin{lemma} (see \cite{Brezis2018GN}.) \label{lemmaGN}
Let the real number $s, s_1, s_2, \gamma \in(0, 1) $ and $1\leq q, q_1, q_2\leq\infty $ such that
$$s=\gamma s_1+(1-\gamma)s_2, \quad \frac{1}{q}=\frac{1}{q_1}+\frac{1}{q_1}.$$
Then
$$ \| f\|_{W^{s, q}(\Omega)}\leq C \| f\|_{W^{s_1, q_1}(\Omega)}^{\gamma} \| f\|_{W^{s_2, q_2}(\Omega)}^{1-\gamma}, \quad \forall f\in W^{s_1, q_1}\cap W^{s_2, q_2}$$
 holds if and only if the integer $s_2\geq 1, q_2=1$ and $s_2-s_1\leq 1-\frac{1}{q_1}$ is not true.
Where $C>0$ is a constat dependent on $\Omega, \beta, s_1, s_2, q_1, q_2$ and the standard domain $\Omega$ is either $\mathbb{R}^d$ or a half space or a Lipschitz bounded domain in $\mathbb{R}^d$.
\end{lemma}
Then we introduce the major tool derived from the above lemmas and Minkowski's inequality.

\begin{lemma} \label{lemmage}
For $s, p, q\in (\frac{1}{2}, 1]$, suppose that $f, g, h, \Lambda_i^{s}f, \Lambda_j^{p}g$ and $ \Lambda_k^{q}h $ are all in $L^2(\mathbb{R}^3)$. Then
%
%
%
where $i, j, k\in \{1, 2, 3\}$ satisfy $\varepsilon_{ijk}\neq0$ and $\varepsilon_{ijk}$ denotes the Ricci symbol.
\end{lemma}

\begin{lemma} \label{lemmap}
Assume that the right-hand sides of there inequalities are all bounded with $m, n\in \mathbb{N}$ and $s\in (0, 1)$. Then
\begin{align*}
\|\partial_{i}^{m+1}\partial_{j}^{n}\Lambda_j^{s} f\|&\leq \| \partial_j^{n+1}\partial_i^{m}\Lambda_i^{s} f\|^{s} \| \partial_j^{n}\partial_i^{m+1}\Lambda_i^{s} f\|^{1-s}\leq \| \Lambda_i^{s} f\|_{H^{m+n+1}}.
\end{align*}
\end{lemma}
\begin{proof}
By using Plancherel theorem, we have
\begin{align*}
 \|\partial_{i}^{m+1}\partial_{j}^{n}\Lambda_j^{s} f\|&= \|\xi_{i}^{m+1}\xi_{j}^{n} \xi_j^{s} \hat{f}\|\\
 &= \|\xi_i^{ms} \xi_i^{s^2} \xi_j^{(n+1)s} \hat{f}^{s} ~\cdot~ \xi_i^{m(1-s)} \xi_i^{(1+s)(1-s)} \xi_j^{n(1-s)} \hat{f}^{1-s}\|\\
 &\leq \|\xi_i^{m}\xi_j^{n+1}\xi_i^{s}\hat{f}\|^{s} ~\|\xi_i^{(m+1)} \xi_j^{n} \xi_i^{s} \hat{f}\|^{1-s}\\
 &= \|\partial_i^{m} \partial_j^{n+1}\Lambda_i^{s} f\|^{s} ~\| \partial_i^{m+1} \partial_j^{n}\Lambda_i^{s} f\|^{1-s}\leq \| \Lambda_i^{s} f\|_{H^{m+n+1}}.
\end{align*}
\end{proof}

\section{Proof of Theorem \ref{theoreml}.} \label{s3}
This section discloses the stability of the system (\ref{bhe}), as shown in Theorem \ref{theoreml}.
The important technique here is bootstrap argument, which first establishes a priori estimate and then applies this method to obtain inequality (\ref{et}).

\begin{proof}
First, we make $L^2$ estimate.
Taking the $L^2$ inner product of the equations (\ref{bhe}) with $(u, b)$, we obtain
\begin{align*}
\frac{1}{2}\frac{d}{dt}\|(u, b)\|^2&+\|(\nu_1^{\frac{1}{2}}\Lambda_1^{\alpha}, \nu_2^{\frac{1}{2}}\Lambda_2^{\alpha})u\|^2+\sigma\nu_3 \|\Lambda_3^{\alpha}u\|^2\\
&+\mu\left(\|(\Lambda_2^{\beta}, \Lambda_3^{\beta})b_1\|^2+\|(\Lambda_1^{\beta}, \Lambda_3^{\beta})b_2\|^2+\|(\Lambda_1^{\beta}, \Lambda_2^{\beta})b_3\|^2\right)=0.
\end{align*}
Integrating in time one has
\begin{multline*}
\|(u, b)\|^2+\int_0^t \|(\nu_1^{\frac{1}{2}}\Lambda_1^{\alpha}, \nu_2^{\frac{1}{2}}\Lambda_2^{\alpha})u\|^2 d\tau+\sigma\nu_3\int_0^t \|\Lambda_3^{\alpha}u\|^2d \tau\\
+\mu\int_0^t\|(\Lambda_2^{\beta}, \Lambda_3^{\beta})b_1\|^2+\|(\Lambda_1^{\beta}, \Lambda_3^{\beta})b_2\|^2+\|(\Lambda_1^{\beta}, \Lambda_2^{\beta})b_3\|^2 d \tau=0.
\end{multline*}
Applying $\partial_i^3$ to (\ref{bhe}) and taking $L^2$ inner product with $(\partial_i^3 u, \partial_i^3 b)$ and then taking summation for $i$ from $1$ to $3$, we obtain
\begin{align*}
\frac{1}{2}\frac{d}{d t}&\|(u, b)\|_{H^3}^2+\|(\nu_1^{\frac{1}{2}}\Lambda_1^{\alpha}, \nu_2^{\frac{1}{2}}\Lambda_2^{\alpha})u\|_{H^3}^2+\sigma\nu_3 \|\Lambda_3^{\alpha}u\|_{H^3}^2\\ &~~+\mu\left(\|(\Lambda_2^{\beta}, \Lambda_3^{\beta})b_1\|_{H^3}^2+\|(\Lambda_1^{\beta}, \Lambda_3^{\beta})b_2\|_{H^3}^2+\|(\Lambda_1^{\beta}, \Lambda_2^{\beta})b_3\|_{H^3}^2\right)\\
&\leq-\sum_{i=1}^3 \int \partial_i^3(u\cdot\nabla u)\cdot\partial_i^3 u dx-\sum_{i=1}^3 \int \partial_i^3(u\cdot\nabla b)\cdot\partial_i^3 b dx\\
&~~+\sum_{i=1}^3 \int \partial_i^3(b\cdot\nabla b)\cdot\partial_i^3 u dx+\sum_{i=1}^3 \int \partial_i^3(b\cdot\nabla u)\cdot\partial_i^3 b dx\\
&:=G_1+H_1+G_2+H_2.
\end{align*}
We notice a fact that
\begin{align*}
\sum_{i=1}^3 \int \partial_i^3 \partial_3 b\cdot \partial_i^3 u dx+\sum_{i=1}^3 \int \partial_i^3 \partial_3 u\cdot\partial_i^3 b dx =0.
\end{align*}
\noindent\underline{\textbf {Estimate of $G_1$.}}
For $i\in\{1, 2, 3\}$, the term $G_1$ can be expanded under the form
\begin{align*}
G_1&=-\sum_{i=1}^3 \int \partial_i^3 u\cdot\nabla u\cdot \partial_i^3 u dx-3\sum_{i=1}^2 \int \partial_{i}^3 u\cdot\nabla \partial_i u\cdot \partial_i^3 u dx
\\&~~~-3\sum_{i=1}^3 \int \partial_i u\cdot\nabla \partial_{i}^2 u\cdot \partial_i^3 u dx-\sum_{i=1}^3 \int u\cdot\nabla \partial_i^3 u\cdot\partial_i^3 u dx\\
&:=I_1+3I_2+3I_3+I_4.
\end{align*}
Here we have the fact that $div~u=0$ to obtain the last term $I_4=0$.

\noindent\textbf {Estimate of $I_1$.} Taking full advantage of anisotropic dissipation, we decompose it into three terms also by using the divergence free condition $div~u=0$ as follows
\begin{align*}
I_{1}=&-\sum_{i=1}^2 \int \partial_i^3 u\cdot\nabla u\cdot \partial_i^3 u dx -\sum_{j=1}^2 \int \partial_3^3 u_{j}\partial_{j} u\cdot \partial_3^3 u dx-\int \partial_3^3 u_{3}\partial_{3} u\cdot \partial_3^3 u dx :=I_{11}+I_{12}+I_{13}.
\end{align*}
Using (ii, iii) of Lemma \ref{lemmage} and Young inequality, we get
\begin{align*}
I_{11}\leq&C\sum_{i=1}^2\|\partial_{i}^3 u\|^{2-\frac{1}{\alpha}}\| \Lambda_1^{\alpha}\partial_{i}^3 u\|^\frac{1}{2\alpha}\| \Lambda_2^{\alpha}\partial_{i}^3 u\|^\frac{1}{2\alpha}\|\nabla u\|^\frac{1}{2}\|\partial_3 \nabla u\|^\frac{1}{2}\\
\leq&C\| u\|_{H^3}\sum_{i=1}^2\|\Lambda_i^{\alpha}\partial_{i}^3 u\|^{(2-\frac{1}{\alpha})\alpha}\|\Lambda_i^{\alpha+1}\partial_{i}^3 u\|^{(2-\frac{1}{\alpha})(1-\alpha)}\| \Lambda_1^{\alpha}\partial_{i}^3 u\|^\frac{1}{2\alpha}\| \Lambda_2^{\alpha}\partial_{i}^3 u\|^\frac{1}{2\alpha}\\
\leq&C\| u\|_{H^3}\|(\Lambda_1^\alpha, \Lambda_{2}^\alpha) u\|_{H^3}^2.
\end{align*}
The term $I_{12}$ can be controlled by utilizing Lemma \ref{lemmap}, Lemma \ref{lemmage} (i, ii, iv) and Young inequality
\begin{align*}
I_{12}\leq&C\sum_{j=1}^2\|\partial_{3}^3 u\|^{1-\frac{1}{2\alpha}}\| \Lambda_1^{\alpha}\partial_{3}^3 u\|^\frac{1}{2\alpha}\| \|\partial_{3}^3 u_{j}\|^{1-\frac{1}{2\alpha}} \|\Lambda_2^{\alpha}\partial_{3}^3 u_{j}\|^\frac{1}{2\alpha}\|\partial_{j} u\|^{1-\frac{1}{2\alpha}}s \|\partial_{j} \Lambda_3^{\alpha} u\|^\frac{1}{2\alpha}\\
\leq&C\sum_{j=1}^2\| u\|_{H^3}^{2-\frac{1}{\alpha}}\| \Lambda_1^{\alpha} u\|_{H^3}^\frac{1}{2\alpha}\| \Lambda_2^{\alpha} u\|_{H^3}^\frac{1}{2\alpha}\|\partial_{j} u\|^{1-\frac{1}{2\alpha}}\|\Lambda_{j}^{\alpha} u\|_{H^1}^{\frac{1}{2\alpha}}\\
\leq&C\sum_{i=1}^2\| u\|_{H^3}^{2-\frac{1}{\alpha}}\| \Lambda_1^{\alpha} u\|_{H^3}^\frac{1}{2\alpha}\| \Lambda_2^{\alpha} u\|_{H^3}^\frac{1}{2\alpha}\|\partial_{j} u\|^{(1-\frac{1}{2\alpha})(\frac{1}{\alpha}-1)}\|\partial_{j} u\|^{(1-\frac{1}{2\alpha})(2-\frac{1}{\alpha})}\\
&\times\|\Lambda_{j}^{\alpha} u\|_{H^1}^{\frac{1}{2\alpha}(\frac{1}{\alpha}-1)}\|\Lambda_{j}^{\alpha} u\|_{H^1}^{\frac{1}{2\alpha}(2-\frac{1}{\alpha})}\\
\leq&C\sum_{j=1}^2\| u\|_{H^3}^{2-\frac{1}{\alpha}}\| \Lambda_1^{\alpha} u\|_{H^3}^\frac{1}{2\alpha}\| \Lambda_2^{\alpha} u\|_{H^3}^\frac{1}{2\alpha}\|\partial_{j} u\|^{(1-\frac{1}{2\alpha})(\frac{1}{\alpha}-1)}\|\Lambda_{j}^{\alpha} u\|^{(1-\frac{1}{2\alpha})(2-\frac{1}{\alpha})\alpha} \\
&\times\|\Lambda_{j}^{\alpha+1} u\|^{(1-\frac{1}{2\alpha})(2-\frac{1}{\alpha})(1-\alpha)}\|\Lambda_{j} u\|_{H^1}^{\frac{1}{2\alpha}(\frac{1}{\alpha}-1)\alpha}\| u\|_{H^1}^{\frac{1}{2\alpha}(\frac{1}{\alpha}-1)(1-\alpha)}\|\Lambda_{j}^{\alpha} u\|_{H^1}^{\frac{1}{2\alpha}(2-\frac{1}{\alpha})}\\
\leq&C\sum_{j=1}^2\|u\|_{H^3}\| \Lambda_1^{\alpha} u\|_{H^3}^\frac{1}{2\alpha}\| \Lambda_2^{\alpha} u\|_{H^3}^\frac{1}{2\alpha}\|\Lambda_{j}^{\alpha} u\|_{H^3}^{2-\frac{1}{\alpha}}\leq C\| u\|_{H^3}\|(\Lambda_1^\alpha, \Lambda_{2}^\alpha) u\|_{H^3}^2.
\end{align*}
Thanks to the divergence-free condition $div~u=0$, $I_{13}$ read as follows and then we use again (i), (ii) and (iv) of Lemma \ref{lemmage} and Young inequality to obtain that
\begin{align*}
I_{13}=&-\int \partial_3^3 u_{3}\partial_{3} u\cdot \partial_3^3 u dx=\sum_{i=1}^2 \int \partial_3^2\partial_{i} u_{i}\partial_3 u \partial_3^3 u dx\\
\leq&C\sum_{i=1}^2\|\partial_{3}^3 u\|^{1-\frac{1}{2\alpha}}\| \Lambda_1^{\alpha}\partial_{3}^3 u\|^\frac{1}{2\alpha}\|\partial_{3} u\|^{1-\frac{1}{2\alpha}}\|\partial_{3} \Lambda_2^{\alpha} u\|^\frac{1}{2\alpha}\| \|\partial_3^2\partial_{i} u_{i}\|^{1-\frac{1}{2\alpha}} \|\Lambda_3^{\alpha}\partial_3^2\partial_{i} u_{i}\|^\frac{1}{2\alpha}\\
\leq&C\sum_{i=1}^2\|\partial_{3}^3 u\|^{1-\frac{1}{2\alpha}}\| \Lambda_1^{\alpha}\partial_{3}^3 u\|^\frac{1}{2\alpha}\|\partial_{3} u\|^{1-\frac{1}{2\alpha}}\|\partial_{3} \Lambda_2^{\alpha} u\|^{\frac{1}{2\alpha}(\frac{1}{\alpha}-1)}\|\partial_{3} \Lambda_2^{\alpha} u\|^{\frac{1}{2\alpha}(2-\frac{1}{\alpha})} \\&\times\|\partial_3^2\partial_{i} u_{i}\|^{(1-\frac{1}{2\alpha})(\frac{1}{\alpha}-1)} \|\partial_3^2\partial_{i} u_{i}\|^{(1-\frac{1}{2\alpha})(2-\frac{1}{\alpha})} \|\Lambda_{i}^{\alpha}u_{i}\|_{H^3}^\frac{1}{2\alpha}\\
\leq&C\sum_{i=1}^2\| u\|_{H^3}^{2-\frac{1}{\alpha}+(1-\frac{1}{2\alpha})(\frac{1}{\alpha}-1)}\| \Lambda_1^{\alpha} u\|_{H^3}^\frac{1}{2\alpha}\|\partial_{3} \Lambda_2 u\|^{\frac{1}{2\alpha}(\frac{1}{\alpha}-1)\alpha}\|\partial_{3} u\|^{\frac{1}{2\alpha}(\frac{1}{\alpha}-1)(1-\alpha)} \|\Lambda_2^{\alpha} u\|_{H^3}^{\frac{1}{2\alpha}(2-\frac{1}{\alpha})} \\&\times\|\Lambda_{i}^{\alpha}\partial_{3}^2 u_{i}\|^{(1-\frac{1}{2\alpha})(2-\frac{1}{\alpha})\alpha} \|\Lambda_{i}^{\alpha}\partial_{3}^2 u_{i}\|^{(1-\frac{1}{2\alpha})(2-\frac{1}{\alpha})(1-\alpha)} \|\Lambda_{i}^{\alpha}u_{i}\|_{H^3}^\frac{1}{2\alpha}\\
\leq&C\sum_{i=1}^2\|u\|_{H^3}\| \Lambda_1^{\alpha} u\|_{H^3}^\frac{1}{2\alpha}\| \Lambda_2^{\alpha} u\|_{H^3}^{\frac{1}{2\alpha}+\frac{1}{2\alpha}(2-\frac{1}{\alpha})}\|\Lambda_{i}^{\alpha} u\|_{H^3}^{(1-\frac{1}{2\alpha})(2-\frac{1}{\alpha})}
\leq C\| u\|_{H^3}\|(\Lambda_1^\alpha, \Lambda_{2}^\alpha) u\|_{H^3}^2.
\end{align*}

\noindent\textbf {Estimate of $I_2$.} We do the following decomposition
\begin{align*}
I_{2}&=-\sum_{i=1}^2 \int \partial_i^2 u\cdot\nabla\partial_i u\cdot \partial_{i}^3 u dx -\sum_{j=1}^2 \int\partial_3^2 u_{j}\partial_{3j} u\cdot \partial_{3}^3 u dx+\sum_{i=1}^2\int \partial_i u_{i}\partial_{3} u\cdot \partial_{3}^3 u dx \\&:=I_{21}+I_{22}+I_{23}.
\end{align*}
The first term $I_{21}$ can be controlled similarly as $I_{11}$,
\begin{align*}
I_{21}\leq&C\sum_{i=1}^2\|\partial_{i}^3 u\|^{1-\frac{1}{2\alpha}}\| \Lambda_1^{\alpha}\partial_{i}^3 u\|^\frac{1}{2\alpha}\|\nabla\partial_{i} u\|^{1-\frac{1}{2\alpha}}\| \Lambda_2^{\alpha}\nabla\partial_{i} u\|^\frac{1}{2\alpha}\|\partial_{i}^2 u\|^\frac{1}{2}\|\partial_{3}\partial_{i}^2  u\|^\frac{1}{2}\\
\leq&C\| u\|_{H^3}\|(\Lambda_1^\alpha, \Lambda_{2}^\alpha) u\|_{H^3}^2.
\end{align*}
For the term $I_{22}$, from Lemma \ref{lemmap}, Lemma \ref{lemmage} (i, ii, iv) and Young inequality, we find
\begin{align*}
I_{22}\leq&C\sum_{j=1}^2\|\partial_{3}^3 u\|^{1-\frac{1}{2\alpha}}\| \Lambda_1^{\alpha}\partial_{3}^3 u\|^\frac{1}{2\alpha} \|\partial_{3}^2 u_{j}\|^{1-\frac{1}{2\alpha}} \|\Lambda_2^{\alpha}\partial_{3}^2 u_{j}\|^\frac{1}{2\alpha}\|\partial_{3j} u\|^{1-\frac{1}{2\alpha}} \|\Lambda_3^{\alpha}\partial_{3j} u\|^\frac{1}{2\alpha}\\
\leq&C\sum_{j=1}^2\| u\|_{H^3}^{2-\frac{1}{\alpha}}\| \Lambda_1^{\alpha} u\|_{H^3}^\frac{1}{2\alpha}\|\Lambda_2^{\alpha}\partial_{3}^2 u_{j}\|^{\frac{1}{2\alpha}(\frac{1}{\alpha}-1)}\|\Lambda_2^{\alpha}\partial_{3}^2 u_{j}\|^{\frac{1}{2\alpha}(2-\frac{1}{\alpha})} \|\partial_{3j} u\|^{(1-\frac{1}{2\alpha})(\frac{1}{\alpha}-1)}
\\&\times\|\partial_{3j} u\|^{(1-\frac{1}{2\alpha})(2-\frac{1}{\alpha})}\|\Lambda_j^{\alpha} u\|_{H^3}^\frac{1}{2\alpha}\\
\leq&C\sum_{j=1}^2\| u\|_{H^3}^{2-\frac{1}{\alpha}}\| \Lambda_1^{\alpha} u\|_{H^3}^\frac{1}{2\alpha}\|\Lambda_2\partial_{3}^2 u_{j}\|^{\frac{1}{2\alpha}(\frac{1}{\alpha}-1)\alpha}\|\partial_{3}^2 u_{j}\|^{\frac{1}{2\alpha}(\frac{1}{\alpha}-1)(1-\alpha)} \|\Lambda_2^{\alpha}\partial_{3}^2 u_{j}\|^{\frac{1}{2\alpha}(2-\frac{1}{\alpha})}
\\&\times \|\partial_{3j} u\|^{(1-\frac{1}{2\alpha})(\frac{1}{\alpha}-1)} \|\Lambda_{j}^{\alpha}\partial_{3} u\|^{(1-\frac{1}{2\alpha})(2-\frac{1}{\alpha})\alpha}\|\Lambda_{j}^{\alpha+1}\partial_{3} u\|^{(1-\frac{1}{2\alpha})(2-\frac{1}{\alpha})(1-\alpha)}\|\Lambda_j^{\alpha} u\|_{H^3}^\frac{1}{2\alpha}\\
\leq&C\sum_{j=1}^2\|u\|_{H^3}\| \Lambda_1^{\alpha} u\|_{H^3}^\frac{1}{2\alpha}\| \Lambda_2^{\alpha} u\|_{H^3}^{\frac{1}{2\alpha}(2-\frac{1}{\alpha})}\|\Lambda_{j}^{\alpha} u\|_{H^3}^{\frac{1}{2\alpha}+(1-\frac{1}{2\alpha})(2-\frac{1}{\alpha})}\\
\leq& C\|u\|_{H^3}\|(\Lambda_1^\alpha, \Lambda_{2}^\alpha) u\|_{H^3}^2.
\end{align*}
For the last term $I_{23}$, we have
\begin{align*}
I_{23}\leq&C\sum_{i=1}^2\|\partial_{3}^3 u\|^{1-\frac{1}{2\alpha}}\| \Lambda_1^{\alpha}\partial_{3}^3 u\|^\frac{1}{2\alpha}\|\partial_{3}^2 u\|^{1-\frac{1}{2\alpha}}\| \Lambda_2^{\alpha}\partial_{3}^2 u\|^\frac{1}{2\alpha}\|\partial_{3i} u_{i}\|^{1-\frac{1}{2\alpha}} \|\Lambda_3^{\alpha}\partial_{3i} u_{i}\|^\frac{1}{2\alpha} \\
\leq&C\sum_{i=1}^2\| u\|_{H^3}^{2-\frac{1}{\alpha}}\| \Lambda_1^{\alpha} u\|_{H^3}^\frac{1}{2\alpha}\| \Lambda_2^{\alpha} u\|_{H^3}^\frac{1}{2\alpha}\|\partial_{3i} u_{i}\|^{(1-\frac{1}{2\alpha})(\frac{1}{\alpha}-1)}\|\partial_{3i} u_{i}\|^{(1-\frac{1}{2\alpha})(2-\frac{1}{\alpha})}
\\&\times\|\Lambda_{i}^{\alpha} u_{i}\|_{H^2}^{\frac{1}{2\alpha}(\frac{1}{\alpha}-1)}\|\Lambda_{i}^{\alpha} u_{i}\|_{H^2}^{\frac{1}{2\alpha}(2-\frac{1}{\alpha})}\\
\leq&C\sum_{i=1}^2\| u\|_{H^3}^{2-\frac{1}{\alpha}}\| \Lambda_1^{\alpha} u\|_{H^3}^\frac{1}{2\alpha}\| \Lambda_2^{\alpha} u\|_{H^3}^\frac{1}{2\alpha}\|\partial_{3i} u_{i}\|^{(1-\frac{1}{2\alpha})(\frac{1}{\alpha}-1)}\|\Lambda_{i}^{\alpha} \partial_{3}u_{i}\|^{(1-\frac{1}{2\alpha})(2-\frac{1}{\alpha})\alpha}
\\&\times\|\Lambda_{i}^{\alpha+1}\partial_{3} u_{i}\|^{(1-\frac{1}{2\alpha})(2-\frac{1}{\alpha})(1-\alpha)}\|\Lambda_{i} u_{i}\|_{H^2}^{\frac{1}{2\alpha}(\frac{1}{\alpha}-1)\alpha}  \| u_{i}\|_{H^2}^{\frac{1}{2\alpha}(\frac{1}{\alpha}-1)(1-\alpha)}\|\Lambda_{i}^{\alpha} u_{i}\|_{H^2}^{\frac{1}{2\alpha}(2-\frac{1}{\alpha})}\\
\leq&C\sum_{i=1}^2\|u\|_{H^3}\| \Lambda_1^{\alpha} u\|_{H^3}^\frac{1}{2\alpha}\| \Lambda_2^{\alpha} u\|_{H^3}^\frac{1}{2\alpha}\|\Lambda_{i}^{\alpha} u\|_{H^3}^{2-\frac{1}{\alpha}}\leq C\| u\|_{H^3}\|(\Lambda_1^\alpha, \Lambda_{2}^\alpha) u\|_{H^3}^2.
\end{align*}
\noindent\textbf {Estimate of $I_3$.} The term $I_3$ can be decomposed into three parts
\begin{align*}
I_{3}&=-\sum_{i=1}^2 \int \partial_i u\cdot\nabla\partial_i^2 u\cdot \partial_i^3 u dx -2\sum_{j=1}^2 \int\partial_3 u_{j}\partial_{j}\partial_3^2 u\cdot \partial_i^3 u dx
+\sum_{i=1}^2\int \partial_i u_{i}\partial_{3}^3 u\cdot \partial_3^3 u dx\\
&:=I_{31}+2I_{32}+I_{33}.
\end{align*}
To estimate the term $I_{31}$, we can refer to $I_{11}$. Namely,
\begin{align*}
I_{31}\leq&C\sum_{i=1}^2\|\partial_{i}^3 u\|^{1-\frac{1}{2\alpha}}\| \Lambda_1^{\alpha}\partial_{i}^3 u\|^\frac{1}{2\alpha}\|\nabla\partial_{i}^2 u\|^{1-\frac{1}{2\alpha}}\| \Lambda_2^{\alpha}\nabla\partial_{i}^2 u\|^\frac{1}{2\alpha}\|\partial_{i} u\|^\frac{1}{2}\|\partial_{3i}  u\|^\frac{1}{2}\\
\leq&C\|u\|_{H^3}\|(\Lambda_1^\alpha,\Lambda_{2}^\alpha) u\|_{H^3}^2.
\end{align*}
For the term $I_{32}$, using Lemma \ref{lemmap}, Lemma \ref{lemmage} (i, ii, iv) and Young inequality, we obtain
\begin{align*}
I_{32}\leq&C\sum_{j=1}^2\|\partial_{3}^2 u\|^{1-\frac{1}{2\alpha}}\| \Lambda_1^{\alpha}\partial_{3}^2 u\|^\frac{1}{2\alpha} \|\partial_{3} u_{j}\|^{1-\frac{1}{2\alpha}} \|\Lambda_2^{\alpha}\partial_{3} u_{j}\|^\frac{1}{2\alpha}\|\partial_{3}^2\partial_{j} u\|^{1-\frac{1}{2\alpha}}  \|\Lambda_3^{\alpha}\partial_{3}^2\partial_{j} u\|^\frac{1}{2\alpha}\\
\leq&C\sum_{i=j}^2\| u\|_{H^3}^{2-\frac{1}{\alpha}}\| \Lambda_1^{\alpha} u\|_{H^3}^\frac{1}{2\alpha}\|\Lambda_2^{\alpha}\partial_{3} u_{j}\|^{\frac{1}{2\alpha}(\frac{1}{\alpha}-1)}\|\Lambda_2^{\alpha}\partial_{3} u_{j}\|^{\frac{1}{2\alpha}(2-\frac{1}{\alpha})}   \|\partial_{3}^2\partial_{j} u\|^{(1-\frac{1}{2\alpha})(\frac{1}{\alpha}-1)}
\\&\times \|\partial_{3}^2\partial_{j} u\|^{(1-\frac{1}{2\alpha})(2-\frac{1}{\alpha})}\|\Lambda_2^{\alpha} u\|_{H^2}^\frac{1}{2\alpha}\\
\leq&C\sum_{i=j}^2\| u\|_{H^3}^{2-\frac{1}{\alpha}}\| \Lambda_1^{\alpha} u\|_{H^3}^\frac{1}{2\alpha}\|\Lambda_2\partial_{3} u_{j}\|^{\frac{1}{2\alpha}(\frac{1}{\alpha}-1)\alpha}\|\partial_{3} u_{j}\|^{\frac{1}{2\alpha}(\frac{1}{\alpha}-1)(1-\alpha)}\|\Lambda_2^{\alpha}\partial_{3} u_{j}\|^{\frac{1}{2\alpha}(2-\frac{1}{\alpha})}
\\&\times \|\partial_{3}^2\partial_{j} u\|^{(1-\frac{1}{2\alpha})(\frac{1}{\alpha}-1)}  \|\Lambda_{j}^{\alpha}\partial_{3}^2 u\|^{(1-\frac{1}{2\alpha})(2-\frac{1}{\alpha})\alpha} \|\Lambda_{j}^{\alpha+1}\partial_{3}^2 u\|^{(1-\frac{1}{2\alpha})(2-\frac{1}{\alpha})(1-\alpha)}\|\Lambda_2^{\alpha} u\|_{H^3}^\frac{1}{2\alpha}\\
\leq&C\sum_{j=1}^2\|u\|_{H^3}\| \Lambda_1^{\alpha} u\|_{H^3}^\frac{1}{2\alpha}\| \Lambda_2^{\alpha} u\|_{H^3}^{\frac{1}{2\alpha}+\frac{1}{2\alpha}(2-\frac{1}{\alpha})}\|\Lambda_{j}^{\alpha} u\|_{H^3}^{(1-\frac{1}{2\alpha})(2-\frac{1}{\alpha})}
\leq C\| u\|_{H^3}\|(\Lambda_1^\alpha,\Lambda_{2}^\alpha) u\|_{H^3}^2.
\end{align*}
For the last term $I_{33}$, by using again Lemma \ref{lemmap}, Lemma \ref{lemmage} (i, ii, iv) and Young inequality that
\begin{align*}
I_{23}\leq&C\sum_{i=1}^2 \|\partial_{3}^3 u\|^{2-\frac{1}{\alpha}} \|\Lambda_1^{\alpha}\partial_{3}^3 u\|^\frac{1}{2\alpha} \|\Lambda_2^{\alpha}\partial_{3}^3 u\|^\frac{1}{2\alpha} \|\partial_{i} u_{i}\|^{1-\frac{1}{2\alpha}} \|\Lambda_3^{\alpha}\partial_{i} u_{i}\|^\frac{1}{2\alpha}  \\
\leq&C\sum_{i=1}^2 \| u\|_{H^3}^{2-\frac{1}{\alpha}} \|\Lambda_1^{\alpha} u\|_{H^3}^\frac{1}{2\alpha}\| \Lambda_2^{\alpha} u\|_{H^3}^\frac{1}{2\alpha}\|\partial_{i} u_{i}\|^{(1-\frac{1}{2\alpha})(\frac{1}{\alpha}-1)}\|\partial_{i} u_{i}\|^{(1-\frac{1}{2\alpha})(2-\frac{1}{\alpha})}
\\&\times\|\Lambda_{i}^{\alpha} u_{i}\|_{H^1}^{\frac{1}{2\alpha}(\frac{1}{\alpha}-1)}\|\Lambda_{i}^{\alpha} u_{i}\|_{H^1}^{\frac{1}{2\alpha}(2-\frac{1}{\alpha})}\\
\leq&C\sum_{i=1}^2\|u\|_{H^3}^{2-\frac{1}{\alpha}} \|\Lambda_1^{\alpha} u\|_{H^3}^\frac{1}{2\alpha} \|\Lambda_2^{\alpha} u\|_{H^3}^\frac{1}{2\alpha} \|\partial_{i} u_{i}\|^{(1-\frac{1}{2\alpha})(\frac{1}{\alpha}-1)} \|\Lambda_{i}^{\alpha} u_{i}\|^{(1-\frac{1}{2\alpha})(2-\frac{1}{\alpha})\alpha}
\\&\times\|\Lambda_{i}^{\alpha+1} u_{i}\|^{(1-\frac{1}{2\alpha})(2-\frac{1}{\alpha})(1-\alpha)} \|\Lambda_{i} u_{i}\|_{H^1}^{\frac{1}{2\alpha}(\frac{1}{\alpha}-1)\alpha}  \| u_{i}\|_{H^1}^{\frac{1}{2\alpha}(\frac{1}{\alpha}-1)(1-\alpha)}\|\Lambda_{i}^{\alpha} u_{i}\|_{H^1}^{\frac{1}{2\alpha}(2-\frac{1}{\alpha})}\\
\leq&C\sum_{i=1}^2\|u\|_{H^3}\| \Lambda_1^{\alpha} u\|_{H^3}^\frac{1}{2\alpha}\| \Lambda_2^{\alpha} u\|_{H^3}^\frac{1}{2\alpha}\|\Lambda_{i}^{\alpha} u\|_{H^3}^{2-\frac{1}{\alpha}}\leq C\| u\|_{H^3}\|(\Lambda_1^\alpha,\Lambda_{2}^\alpha) u\|_{H^2}^2.
\end{align*}
Combining these estimates, we finally get
\[
G_{1}\leq C\| u\|_{H^3}\|(\Lambda_1^\alpha, \Lambda_{2}^\alpha) u\|_{H^3}^2.
\]
\noindent\underline{ \textbf {Estimate of $H_1$.}}
For $i\in\{1, 2, 3\}$, we do the following decomposition on
\begin{align*}
H_1=&-\sum_{i=1}^3 \int \partial_{i}^3 u\cdot\nabla b\cdot \partial_{i}^3 b dx-3\sum_{m=1}^2\sum_{i=1}^3 \int \partial_i^{3-m} u\cdot\nabla \partial_i^{m} b\cdot \partial_{i}^3 b dx:=K_1+3K_2.
\end{align*}
\noindent\textbf {Estimate of $K_1$.} We rewrite it as
\begin{align*}
K_1=-\sum_{i=1}^2 \int \partial_{i}^3 u\cdot\nabla b\cdot \partial_{i}^3 b dx-\int \partial_{3}^3 u\cdot\nabla b\cdot \partial_{3}^3 b dx=K_{11}+K_{12}.
\end{align*}
For the estimate of $k_{11}$, we consider the following two cases.
Case 1: $i=k$. Since the each magnetic equation lacks dissipation in its own direction, take full advantage of anisotropic dissipation, we decompose $K_{111}$ into two terms also by using the divergence free condition $div~b=0$ as follows
\begin{align*}
K_{111}=&\sum_{i=1}^2 \int \partial_{i}^3 u\cdot\nabla b_i\partial_{i}^2\partial_{i'} b_{i'} dx +\sum_{i=1}^2\int \partial_{i}^3 u\cdot\nabla b_i\partial_{i}^2\partial_{i''} b_{i''} dx :=K_{1111}+K_{1112}.
\end{align*}
\noindent Where $i, i', i''\in \{1, 2, 3\}$ satisfy $\varepsilon_{ii'i''}\neq0$.
By using (ii, iii) of Lemma \ref{lemmage} and Young inequality, we get
\begin{align*}
K_{111}\leq&C\sum_{i=1}^2 \|\partial_{i}^3 u\|^{1-\frac{1}{2\alpha}} \|\Lambda_i^{\alpha}\partial_{i}^3 u\|^\frac{1}{2\alpha} \|\partial_{i}^2\partial_{i'} b_{i'}\|^{1-\frac{1}{2\beta}} \|\Lambda_{i''}^{\beta}\partial_{i}^2\partial_{i'} b_{i'}\|^\frac{1}{2\beta} \|\nabla b_i\|^\frac{1}{2} \|\nabla\partial_{i'} b_i\|^\frac{1}{2}\\
\leq&C \|b\|_{H^3}\sum_{i=1}^2 \|\Lambda_i^{\alpha} u\|_{H^3} \|\Lambda_i^{\beta} b_{i'}\|_{H^3}^{1-\frac{1}{2\beta}} \|\Lambda_{i''}^{\beta} b_{i'}\|_{H^3}^{\frac{1}{2\alpha}}\\
\leq &C \|b\|_{H^3}\sum_{i=1}^2 \left(\| \Lambda_i^{\alpha} u\|_{H^3}^2+\|(\Lambda_i^\beta, \Lambda_{i''}^\beta) b_{i'}\|_{H^3}^2\right).
\end{align*}
The last term $K_{112}$ can be estimated similarly as $K_{111}$.
\begin{align*}
K_{112}\leq &C\| b\|_{H^3}\sum_{i=1}^2\left(\| \Lambda_i^{\alpha} u\|_{H^3}^2+\|(\Lambda_i^\beta, \Lambda_{i'}^\beta) b_{i''}\|_{H^3}^2\right).
\end{align*}

\noindent Case 2: $ i\neq k$. we do the following decomposition
\begin{align*}
K_{112}&=-\sum_{i=1}^2\int \partial_{i}^3 u\cdot\nabla b_{i'}\partial_{i}^3 b_{i'} dx-\sum_{i=1}^2\int \partial_{i}^3 u\cdot\nabla b_{i''}\partial_{i}^3 b_{i''} dx :=K_{1121}+K_{1122}.
\end{align*}
Here $K_{1122}$ and $K_{1121}$ can be bounded in the same way and then by utilizing Lemma \ref{lemmage} (ii, iii) and Young inequality, one has the first estimate
\begin{align*}
K_{1121}\leq&C\sum_{i=1}^2\|\partial_{i}^3 u\|^{1-\frac{1}{2\alpha}}\| \Lambda_i^{\alpha}\partial_{i}^3 u\|^\frac{1}{2\alpha}\|\partial_{i}^3 b_{i'}\|^{1-\frac{1}{2\beta}}\| \Lambda_{i''}^{\beta}\partial_{i}^3 b_{i'}\|^\frac{1}{2\beta} \|\nabla b_{i'}\|^\frac{1}{2}\|\partial_{i'}\nabla b_{i'}\|^\frac{1}{2}\\
\leq&C\|b\|_{H^3}\sum_{i=1}^2\|\Lambda_i^{\alpha} u\|_{H^3}\|\Lambda_i^{\beta}b_{i'}\|_{H^3}^{1-\frac{1}{2\beta}}\|\Lambda_{i''}^{\beta}b_{i'}\|_{H^3}^{\frac{1}{2\beta}}\\
\leq &C\| b\|_{H^3}\sum_{i=1}^2\left(\| \Lambda_i^{\alpha} u\|_{H^3}^2+\|(\Lambda_i^\beta, \Lambda_{i''}^\beta) b_{i'}\|_{H^3}^2\right)
\end{align*}
and
\begin{align*}
K_{1122}\leq &C\| b\|_{H^3}\sum_{i=1}^2\left(\| \Lambda_i^{\alpha} u\|_{H^3}^2+\|(\Lambda_i^\beta, \Lambda_{i'}^\beta) b_{i''}\|_{H^3}^2\right).
\end{align*}
Combining these cases, we get the estimate of the $K_{11}, $
\begin{align*}\label{k11}
K_{11}\leq C\|b\|_{H^3}\sum_{i=1}^2\left(\|\Lambda_i^{\alpha} u\|_{H^3}^2+\|(\Lambda_i^\beta, \Lambda_{i''}^\beta) b_{i'}\|_{H^3}^2\right).
\end{align*}

\noindent To estimate the term $K_{12}$, we consider the following five cases
\begin{align*}
1. ~ 3=&k=j ; \quad 2. ~ 3=k\neq j;\quad 3. ~ 3=j\neq k;\quad 4. ~ 3\neq j\neq k; \quad 5. ~ j=k\neq 3.
\end{align*}
Case 1: $3=k=j$. Since the velocity equations lack vertical dissipation and the each magnetic equation lacks dissipation in its own direction, we decompose $K_{121}$ into two terms also by using the divergence free condition $div~u=div~b=0$ as follows
\begin{align*}
K_{121}&=\int \partial_{3}^2\partial_{1} u_{1} \partial_{3} b_{3}\partial_{3}^2\partial_{1} b_{1} dx +\int \partial_{3}^2\partial_{1} u_{1} \partial_{3} b_{3}\partial_{3}^2\partial_{2} b_{2} dx +\int \partial_{3}^2\partial_{2} u_{2} \partial_{3} b_{3}\partial_{3}^2\partial_{1} b_{1} dx
\\&~~~+\int \partial_{3}^2\partial_{2} u_{2} \partial_{3} b_{3}\partial_{3}^2\partial_{2} b_{2} dx \\
&:=K_{1211}+K_{1212}+K_{1213}+K_{1214}.
\end{align*}
For the first term $K_{1211}$, taking advantage of Lemma \ref{lemmage} (ii, iii) and Young inequality, we get
\begin{align*}
K_{1211}\leq&C \|\partial_{3}^2\partial_{1} u_{1}\|^{1-\frac{1}{2\alpha}} \|\Lambda_1^{\alpha}\partial_{3}^2\partial_{1} u_{1}\|^\frac{1}{2\alpha} \|\partial_{3}^2\partial_{1} b_{1}\|^{1-\frac{1}{2\beta}} \|\Lambda_{3}^{\beta}\partial_{3}^2\partial_{1} b_{1}\|^\frac{1}{2\beta} \|\partial_{3} b_{3}\|^\frac{1}{2} \|\partial_{23} b_{3}\|^\frac{1}{2} \\
\leq&C\|b\|_{H^3}\|\Lambda_1^{\alpha} u_{1}\|_{H^3} \|\Lambda_{3}^{\beta} b_{1}\|_{H^3}
\leq C\|b\|_{H^3}(\|\Lambda_1^{\alpha} u_{1}\|_{H^3}^2+ \|\Lambda_{3}^\beta b_{1}\|_{H^3}^2).
\end{align*}
The terms $K_{1212}$, $K_{1213}$ and $K_{1214}$ can be bounded in same way,
\begin{align*}
K_{1212}+K_{1213}+K_{1214}\leq C \|b\|_{H^3} \left(\|(\Lambda_1^{\alpha} u_{1}, \Lambda_2^{\alpha} u_{2})\|_{H^3}^2+\|(\Lambda_{3}^\beta b_{1}, \Lambda_{3}^\beta b_{2})\|_{H^3}^2\right).
\end{align*}
\noindent Case 2: $k=3 $ and $ j\neq3$. We do the following decomposition
\begin{align*}
K_{122}&=\int \partial_{3}^3 u_{1} \partial_{1} b_{3}\partial_{3}^2\partial_{1} b_{1} dx +\int \partial_{3}^3 u_{1} \partial_{1} b_{3}\partial_{3}^2\partial_{2} b_{2} dx +\int \partial_{3}^3 u_{2} \partial_{2} b_{3}\partial_{3}^2\partial_{1} b_{1} dx \\&~~~+\int \partial_{3}^3 u_{2} \partial_{2} b_{3}\partial_{3}^2\partial_{2} b_{2} dx \\
&:=K_{1221}+K_{1222}+K_{1223}+K_{1224}.
\end{align*}
Since the structures of terms $K_{1221}$, $K_{1222}$, $K_{1223}$ and $K_{1224}$ are same, we only estimate $K_{1221}$.
Referring to Lemma \ref{lemmage} (i, ii, iv) and Young inequality, we can infer
\begin{align*}
K_{1221}\leq&\|\partial_{3}^3 u_{1}\|^{1-\frac{1}{2\alpha}} \|\Lambda_1^{\alpha}\partial_{3}^3 u_{1}\|^\frac{1}{2\alpha} \|\partial_{3}^2\partial_{1} b_{1}\|^{1-\frac{1}{2\beta}} \|\Lambda_{3}^{\beta}\partial_{3}^2\partial_{1} b_{1}\|^\frac{1}{2\beta} \|\partial_{1} b_{3}\|^{1-\frac{1}{2\beta}} \|\Lambda_{2}^{\beta}\partial_{1} b_{3}\|^\frac{1}{2\beta} \\
\leq&C \|\partial_{3}^3 u_{1}\|^{1-\frac{1}{2\alpha}} \|\Lambda_1^{\alpha}\partial_{3}^3 u_{1}\|^\frac{1}{2\alpha} \|\partial_{3}^2\partial_{1} b_{1}\|^{(1-\frac{1}{2\beta})(\frac{1}{2\alpha}+\frac{1}{2\beta}-1)} \|\partial_{3}^2\partial_{1} b_{1}\|^{(1-\frac{1}{2\beta})(2-\frac{1}{2\alpha}-\frac{1}{2\beta})} \\&\times \|\Lambda_{3}^{\beta}\partial_{3}^2\partial_{1} b_{1}\|^\frac{1}{2\beta} \|\partial_{1} b_{3}\|^{1-\frac{1}{2\beta}} \|\Lambda_{2}^{\beta}\partial_{1} b_{3}\|^{\frac{1}{2\beta}(\frac{1}{2\alpha}+\frac{1}{2\beta}-1)} \|\Lambda_{2}^{\beta}\partial_{1} b_{3}\|^{\frac{1}{2\beta}(2-\frac{1}{2\alpha}-\frac{1}{2\beta})} \\
\leq&C\| u\|_{H^3}^{1-\frac{1}{2\alpha}}\| \Lambda_1^{\alpha} u_{1}\|_{H^3}^\frac{1}{2\alpha}\|\partial_{3}^2\partial_{1} b_{1}\|^{(1-\frac{1}{2\beta})(\frac{1}{2\alpha}+\frac{1}{2\beta}-1)} \|\Lambda_{3}^{\beta} b_{1}\|_{H^3}^{(1-\frac{1}{2\beta})(2-\frac{1}{2\alpha}-\frac{1}{2\beta})} \\&\times \|\Lambda_{3}^{\beta}\partial_{3}^2\partial_{1} b_{1}\|^\frac{1}{2\beta} \|\partial_{1} b_{3}\|^{1-\frac{1}{2\beta}}  \|b\|_{H^3}^{\frac{1}{2\beta}(\frac{1}{2\alpha}+\frac{1}{2\beta}-1)}\|\Lambda_{2}^{\beta}\partial_{1} b_{3}\|^{\frac{1}{2\beta}(2-\frac{1}{2\alpha}-\frac{1}{2\beta})} \\
\leq& C\|(u, b)\|_{H^3}\left(\|\Lambda_1^{\alpha} u_{1}\|_{H^3}^2+\|(\Lambda_2^\beta b_{3}, \Lambda_{3}^\beta b_{1})\|_{H^3}^2\right)
\end{align*}
and
\begin{align*}
K_{1222}+K_{1223}+K_{1224}\leq &C \|(u,b)\|_{H^3} \\&\times\left(\|(\Lambda_1^{\alpha} u_{1}, \Lambda_2^{\alpha} u_{2})\|_{H^3}^2+\|(\Lambda_1^\beta, \Lambda_{2}^\beta) b_{3}\|_{H^3}^2+\|( \Lambda_{3}^\beta b_{1}, \Lambda_{3}^\beta b_{2})\|_{H^3}^2\right).
\end{align*}
\noindent Case 3: $j=3$ and $k\neq3$. For the term $K_{123}$, we split it into four parts
\begin{align*}
K_{123}&=-\int \partial_3^3 u_3\partial_3 b_{1}\partial_3^3 b_{1} dx-\int \partial_3^3 u_3\partial_3 b_{2}\partial_3^3 b_{2} dx \\
&=\int \partial_3^2\partial_{1} u_1\partial_3 b_{1}\partial_3^3 b_{1} dx+ \int \partial_3^2\partial_{1} u_1\partial_3 b_{2}\partial_3^3 b_{2} dx
\\&~~~+\int \partial_3^2\partial_{2} u_2\partial_3 b_{1}\partial_3^3 b_{1} dx+ \int \partial_3^2\partial_{2} u_2\partial_3 b_{2}\partial_3^3 b_{2} dx\\
&:= K_{1231}+K_{1232}+K_{1233}+K_{1234}.
\end{align*}
Note that the case 3 can be treated by the way of the case 1, and then we obtain
\begin{align*}
K_{123}\leq C\|b\|_{H^3}\left(\|(\Lambda_1^{\alpha}, \Lambda_2^{\alpha})(u_{1}, u_{2})\|_{H^3}^2+\|\left((\Lambda_2^\beta, \Lambda_{3}^\beta) b_{1}, (\Lambda_1^\beta, \Lambda_{3}^\beta) b_{2}\right)\|_{H^3}^2\right).
\end{align*}
\noindent Case 4: $ j\neq k\neq3$. We write
\begin{align*}
K_{124}=&\int \partial_{3}^3 u_{1} \partial_{1} b_{2}\partial_{3}^3 b_{2} dx +\int \partial_{3}^3 u_{2} \partial_{2} b_{1}\partial_{3}^3 b_{1} dx:=K_{1241}+K_{1242}.
\end{align*}
Thanks to Lemma \ref{lemmage} (i, ii, iv) and Young inequality, we obtain
\begin{align*}
K_{1241}\leq&C\|\partial_{3}^3 u_{1}\|^{1-\frac{1}{2\alpha}}\| \Lambda_2^{\alpha}\partial_{3}^3 u_{1}\|^\frac{1}{2\alpha}\|\partial_{3}^3 b_{2}\|^{1-\frac{1}{2\beta}}\| \Lambda_{1}^{\beta}\partial_{3}^3 b_{2}\|^\frac{1}{2\beta} \|\partial_{1} b_{2}\|^{1-\frac{1}{2\beta}}\|\Lambda_{3}^{\beta}\partial_{1} b_{2}\|^\frac{1}{2\beta} \\
\leq&C\| u\|_{H^3}^{1-\frac{1}{2\alpha}}\| \Lambda_2^{\alpha} u_{1}\|_{H^3}^\frac{1}{2\alpha}\|b\|_{H^3}^{1-\frac{1}{2\beta}}\| \Lambda_{1}^{\beta} b_{2}\|_{H^3}^\frac{1}{2\beta}\|\partial_{1} b_{2}\|^{(1-\frac{1}{2\beta})(\frac{1}{2\alpha}+\frac{1}{2\beta}-1)} \\&\times\|\partial_{1} b_{2}\|^{(1-\frac{1}{2\beta})(2-\frac{1}{2\alpha}-\frac{1}{2\beta})}\|\Lambda_{3}^{\beta}\partial_{1} b_{2}\|^{\frac{1}{2\beta}(\frac{1}{2\alpha}+\frac{1}{2\beta}-1)}\|\Lambda_{3}^{\beta}\partial_{1} b_{2}\|^{\frac{1}{2\beta}(2-\frac{1}{2\alpha}-\frac{1}{2\beta})} \\
\leq&C\| u\|_{H^3}^{1-\frac{1}{2\alpha}}\| \Lambda_2^{\alpha} u_{1}\|_{H^3}^\frac{1}{2\alpha}\|b\|_{H^3}^{1-\frac{1}{2\beta}}\| \Lambda_{1}^{\beta} b_{2}\|_{H^3}^\frac{1}{2\beta}\|\partial_{1} b_{2}\|^{(1-\frac{1}{2\beta})(\frac{1}{2\alpha}+\frac{1}{2\beta}-1)} \\&\times\|\Lambda_{1}^{\beta} b_{2}\|_{H^3}^{(1-\frac{1}{2\beta})(2-\frac{1}{2\alpha}-\frac{1}{2\beta})} \|b\|_{H^3}^{\frac{1}{2\beta}(\frac{1}{2\alpha}+\frac{1}{2\beta}-1)}\|\Lambda_{3}^{\beta} b_{2}\|_{H^3}^{\frac{1}{2\beta}(2-\frac{1}{2\alpha}-\frac{1}{2\beta})} \\
\leq& C\|(u, b)\|_{H^3}\left(\|\Lambda_2^{\alpha} u_{1}\|_{H^3}^2+\|(\Lambda_1^\beta, \Lambda_{3}^\beta) b_{2}\|_{H^3}^2\right)
\end{align*}
and
\begin{align*}
K_{1242}\leq C\|(u, b)\|_{H^3}\left(\|\Lambda_1^{\alpha} u_{2}\|_{H^3}^2+\|(\Lambda_1^\beta, \Lambda_{2}^\beta) b_{3}\|_{H^3}^2\right).
\end{align*}
\noindent Case 5: $ j=k\neq 3$. We do the decomposition
\begin{align*}
K_{125}=&\int \partial_{3}^3 u_{1} \partial_{1} b_{1}\partial_{3}^3 b_{1} dx +\int \partial_{3}^3 u_{2} \partial_{2} b_{2}\partial_{3}^3 b_{2} dx:=K_{1251}+K_{1252}.
\end{align*}
The estimate of the term $I_{1251}$ is similar to the term $I_{1221}$, we have
\begin{align*}
K_{1251}\leq&C\|\partial_{3}^3 u_{1}\|^{1-\frac{1}{2\alpha}}\| \Lambda_1^{\alpha}\partial_{3}^3 u_{1}\|^\frac{1}{2\alpha}\|\partial_{3}^3 b_{1}\|^{1-\frac{1}{2\beta}}\| \Lambda_{3}^{\beta}\partial_{3}^3 b_{1}\|^\frac{1}{2\beta} \|\partial_{1} b_{1}\|^{1-\frac{1}{2\beta}}\|\Lambda_{2}^{\beta}\partial_{1} b_{1}\|^\frac{1}{2\beta} \\
\leq&C\|\partial_{3}^3 u_{1}\|^{1-\frac{1}{2\alpha}}\| \Lambda_1^{\alpha}\partial_{3}^3 u_{1}\|^\frac{1}{2\alpha}\|\partial_{3}^3 b_{1}\|^{(1-\frac{1}{2\beta})(\frac{1}{2\alpha}+\frac{1}{2\beta}-1)} \|\partial_{3}^3 b_{1}\|^{(1-\frac{1}{2\beta})(2-\frac{1}{2\alpha}-\frac{1}{2\beta})}\\&\times \|\Lambda_{3}^{\beta}\partial_{3}^3 b_{1}\|^\frac{1}{2\beta}\|\partial_{1} b_{1}\|^{1-\frac{1}{2\beta}} \|\Lambda_{2}^{\beta}\partial_{1} b_{1}\|^{\frac{1}{2\beta}(\frac{1}{2\alpha}+\frac{1}{2\beta}-1)}\|\Lambda_{2}^{\beta}\partial_{1} b_{1}\|^{\frac{1}{2\beta}(2-\frac{1}{2\alpha}-\frac{1}{2\beta})} \\
\leq&C\| u\|_{H^3}^{1-\frac{1}{2\alpha}}\| \Lambda_1^{\alpha} u_{1}\|_{H^3}^\frac{1}{2\alpha}\|\partial_{3}^3 b_{1}\|^{(1-\frac{1}{2\beta})(\frac{1}{2\alpha}+\frac{1}{2\beta}-1)} \|\Lambda_{3}^{\beta} b_{1}\|_{H^3}^{(1-\frac{1}{2\beta})(2-\frac{1}{2\alpha}-\frac{1}{2\beta})} \\&\times \|\Lambda_{3}^{\beta}\partial_{3}^3 b_{1}\|^\frac{1}{2\beta}\|\partial_{1} b_{1}\|^{1-\frac{1}{2\beta}} \| b\|_{H^3}^{\frac{1}{2\beta}(\frac{1}{2\alpha}+\frac{1}{2\beta}-1)}\|\Lambda_{2}^{\beta}\partial_{1} b_{1}\|^{\frac{1}{2\beta}(2-\frac{1}{2\alpha}-\frac{1}{2\beta})} \\
\leq& C\|(u, b)\|_{H^3}\left(\|\Lambda_1^{\alpha} u_{1}\|_{H^3}^2+\|(\Lambda_2^\beta, \Lambda_{3}^\beta) b_{1}\|_{H^3}^2\right)
\end{align*}
and
\begin{align*}
K_{1252}\leq \|(u,b)\|_{H^3}\left(\|\Lambda_2^{\alpha} u_{2}\|_{H^3}^2+\|(\Lambda_1^\beta, \Lambda_{3}^\beta) b_{2}\|_{H^3}^2\right).
\end{align*}
\noindent\textbf {Estimate of $K_2$.} For the term $K_{2}$, we just discuss the three cases.
\begin{align*}
1. ~ i=k\neq3; \quad 2. ~ i=k=3; \quad 3. ~ i\neq k.
\end{align*}
\noindent Case 1: $i=k\neq3$. We decompose $K_{21}$ into two terms also by using the $div~b=0$ as follows
\begin{align*}
K_{21}=&\sum_{i,m=1}^2 \int \partial_i^{3-m} u\cdot\nabla \partial_i^{m} b_i\partial_i^2\partial_{i'} b_{i'} dx +\sum_{i,m=1}^2\int \partial_i^{3-m} u\cdot\nabla \partial_i^{m} b_i\partial_i^2\partial_{i''} b_{i''} dx :=K_{211}+K_{212}.
\end{align*}
To estimate the first term $K_{211}$, we use Lemma \ref{lemmage} (i, ii, iv) and Young inequality to get
\begin{align*}
K_{211}\leq&C\sum_{i,m=1}^2\|\nabla\partial_i^{m} b_i\|^{1-\frac{1}{2\beta}}\|\Lambda_{i'}^{\beta}\nabla\partial_i^{m} b_i\|^\frac{1}{2\beta} \|\partial_i^2\partial_{i'} b_{i'}\|^{1-\frac{1}{2\beta}} \|\Lambda_{i''}^{\beta}\partial_i^2\partial_{i'} b_{i'}\|^\frac{1}{2\beta}  \\ &\times\|\partial_i^{3-m} u\|^{1-\frac{1}{2\alpha}} \|\Lambda_i^{\alpha}\partial_i^{3-m} u\|^\frac{1}{2\alpha}\\
\leq&C\sum_{i,m=1}^2\|\nabla \partial_i^{m} b_i\|^{1-\frac{1}{2\beta}}\|\Lambda_{i'}^{\beta}\nabla\partial_i^{m} b_i\|^\frac{1}{2\beta}\|\partial_i^2\partial_{i'} b_{i'}\|^{(1-\frac{1}{2\beta})(\frac{1}{2\alpha}+\frac{1}{2\beta}-1)} \|\Lambda_{i}^{\beta} b_{i'}\|_{H^3}^{(1-\frac{1}{2\beta})(2-\frac{1}{2\alpha}-\frac{1}{2\beta})}
\\ &\times\|\Lambda_{i''}^{\beta}\partial_i^2\partial_{i'} b_{i'}\|^\frac{1}{2\beta} \|\partial_i^{3-m} u\|^{1-\frac{1}{2\alpha}} \|u\|_{H^3}^{\frac{1}{2\beta}(\frac{1}{2\alpha}+\frac{1}{2\beta}-1)} \|\Lambda_i^{\alpha}\partial_i^{3-m} u\|^{\frac{1}{2\beta}(2-\frac{1}{2\alpha}-\frac{1}{2\beta})}\\
\leq &C\|(u, b)\|_{H^3}\sum_{i=1}^2\left(\| \Lambda_i^{\alpha} u\|_{H^3}^2+\|\Lambda_{i'}^{\beta}b_i\|_{H^3}^2+\|(\Lambda_i^\beta, \Lambda_{i''}^\beta) b_{i'}\|_{H^3}^2\right).
\end{align*}
Similarly,
\[
K_{212}\leq C\|(u, b)\|_{H^3}\sum_{i=1}^2\left(\| \Lambda_i^{\alpha} u\|_{H^3}^2+\|\Lambda_{i''}^{\beta}b_i\|_{H^3}^2+\|(\Lambda_i^\beta, \Lambda_{i'}^\beta) b_{i''}\|_{H^3}^2\right).
\]
Case 2: $i=k=3$. By using the divergence free condition $div~b=0$, we decompose $K_{22}$ into two terms
\begin{align*}
K_{22}=\sum_{m=1}^2\int \partial_{3}^{3-m} u\cdot\nabla\partial_{3}^{m} b_{3}\partial_3^2\partial_{1} b_{1} dx +\sum_{m=1}^2\int \partial_{3}^{3-m} u\cdot\nabla\partial_{3}^{m} b_{3}\partial_3^2\partial_{2} b_{2} dx :=K_{221}+K_{222}.
\end{align*}
To estimate the first term $K_{221}$, it is again sufficient thank to Lemma \ref{lemmage} (i, ii, iv) and Young inequality to obtain that
\begin{align*}
K_{221}\leq&C\sum_{m=1}^2\|\nabla\partial_3^{m} b_3\|^{1-\frac{1}{2\beta}} \|\Lambda_{2}^{\beta}\nabla\partial_3^{m} b_3\|^\frac{1}{2\beta} \|\partial_3^2\partial_{1} b_{1}\|^{1-\frac{1}{2\beta}} \| \Lambda_{3}^{\beta}\partial_3^2\partial_{1} b_{1}\|^\frac{1}{2\beta} \|\partial_3^{3-m} u\|^{1-\frac{1}{2\alpha}} \\ &\times\|\Lambda_1^{\alpha}\partial_3^{3-m} u\|^\frac{1}{2\alpha}\\
\leq&C\sum_{m=1}^2\|\nabla\partial_3^{m} b_3\|^{1-\frac{1}{2\beta}} \|\Lambda_{2}^{\beta}\nabla\partial_3^{m} b_3\|^\frac{1}{2\beta} \|\partial_3^2\partial_{1} b_{1}\|^{(1-\frac{1}{2\beta})(\frac{1}{2\alpha}+\frac{1}{2\beta}-1)}  \|\Lambda_{3}^{\beta} b_{1}\|_{H^3}^{(1-\frac{1}{2\beta})(2-\frac{1}{2\alpha}-\frac{1}{2\beta})}
\\&\times \|\Lambda_{3}^{\beta}\partial_3^2\partial_{1} b_{1}\|^\frac{1}{2\beta}\|\partial_3^{3-m} u\|^{1-\frac{1}{2\alpha}} \|u\|_{H^3}^{\frac{1}{2\beta}(\frac{1}{2\alpha}+\frac{1}{2\beta}-1)} \|\Lambda_1^{\alpha}\partial_3^{3-m} u\|^{\frac{1}{2\beta}(2-\frac{1}{2\alpha}-\frac{1}{2\beta})}\\
\leq &C\|(u, b)\|_{H^3}\left(\|\Lambda_1^{\alpha} u\|_{H^3}^2+\|(\Lambda_3^\beta b_{1}, \Lambda_{2}^\beta b_{3})\|_{H^3}^2\right).
\end{align*}
In the similar way, we have
\[
K_{222}\leq C\|(u, b)\|_{H^3}\left(\| \Lambda_1^{\alpha} u\|_{H^3}^2+\|(\Lambda_3^\beta b_{2}, \Lambda_{2}^\beta b_{3})\|_{H^3}^2\right).
\]
\noindent Case 3: $i\neq k$. We write that
\begin{align*}
K_{23}&=-\sum_{m=1}^2\sum_{i=1}^3\int \partial_i^{3-m} u\cdot\nabla\partial_i^{m} b_{i'}\partial_{i}^3 b_{i'} dx-\sum_{m=1}^2\sum_{i=1}^3\int \partial_i^{3-m} u\cdot\nabla\partial_i^{m} b_{i''}\partial_{i}^3 b_{i''} dx
\\&:=K_{231}+K_{232}.
\end{align*}
It follows by using Lemma \ref{lemmage} (ii, iii) and Young inequality that
\begin{align*}
K_{231}\leq&C\sum_{m=1}^2\sum_{i=1}^3 \|\partial_{i}^3 b_{i'}\|^{1-\frac{1}{2\beta}} \|\Lambda_i^{\beta}\partial_{i}^3 b_{i'}\|^\frac{1}{2\beta} \|\nabla\partial_i^{m} b_{i'}\|^{1-\frac{1}{2\beta}} \|\Lambda_{i''}^{\beta}\nabla\partial_i^{m} b_{i'}\|^\frac{1}{2\beta} \\ &\times\|\partial_i^{3-m} u_{i}\|^\frac{1}{2} \|\partial_{i'}\partial_i^{3-m} u_{i}\|^\frac{1}{2}\\
\leq&C\|u\|_{H^3}\sum_{i=1}^3\|\Lambda_i^{\beta}b_{i'}\|_{H^3}^{2-\frac{1}{2\beta}}\|\Lambda_{i''}^{\beta}b_{i'}\|_{H^3}^{\frac{1}{2\beta}}
\leq C\| u\|_{H^3}\sum_{i=1}^3\|(\Lambda_i^\beta, \Lambda_{i''}^\beta) b_{i'}\|_{H^3}^2.
\end{align*}
We use the same argument to get the estimate
\begin{align*}
K_{232}\leq C\| u\|_{H^3}\sum_{i=1}^3\|(\Lambda_i^\beta, \Lambda_{i'}^\beta) b_{i''}\|_{H^3}^2.
\end{align*}
Combining these estimates, we finally obtain that
\begin{align*}
H_{1}\leq C\|(u, b)\|_{H^3}\left(\| (\Lambda_1^{\alpha}, \Lambda_2^{\alpha}) u\|_{H^3}^2+\sum_{i=1}^3\|(\Lambda_{i'}^\beta, \Lambda_{i''}^\beta) b_{i}\|_{H^3}^2\right).
\end{align*}
\noindent\underline{\textbf {Estimate of $G_2$.}}
We will do the decomposition as follow
\begin{align*}
G_2&=\sum_{i=1}^3 \int \partial_{i}^3 b\cdot\nabla b\cdot\partial_{i}^3 u dx+3\sum_{m=1}^2\sum_{i=1}^3 \int \partial_i^{3-m} b\cdot\nabla\partial_i^{m} b\cdot\partial_{i}^3 u dx+\sum_{i=1}^3 \int b\cdot\nabla \partial_{i}^3 b\cdot\partial_{i}^3 u dx
\\&:=J_{1}+J_{2}+J_{3}.
\end{align*}
\noindent\textbf {Estimate of $J_1$.} We write
\begin{align*}
J_1=\sum_{i=1}^2 \int\partial_{i}^3 b\cdot\nabla b\cdot\partial_{i}^3 u dx+ \int \partial_3^3 b\cdot\nabla b\cdot\partial_3^3 u dx:=J_{11}+J_{12}.
\end{align*}
The term $J_{1}$ have a similar structure to $K_{1}$, we can deduce that
\begin{align*}
J_1\leq C\|(u, b)\|_{H^3}\left(\|(\Lambda_1^{\alpha}, \Lambda_2^{\alpha}) (u_{1}, u_{2})\|_{H^3}^2+\sum_{i=1}^3\|(\Lambda_{i'}^\beta, \Lambda_{i''}^\beta) b_{i}\|_{H^3}^2\right).
\end{align*}
\noindent\textbf {Estimate of $J_2$.} We write
\begin{align*}
J_2=\sum_{i,m=1}^2 \int \partial_i^{3-m} b\cdot\nabla\partial_i^{m} b\cdot \partial_{i}^3 u dx+ \sum_{m=1}^2\int \partial_3^{3-m} b\cdot\nabla\partial_3^{m} u\cdot \partial_3^3 u dx:=J_{21}+J_{22}.
\end{align*}
For the estimate of $J_{21}$, we use the same idea as $K_{11}$. This yields
\[
J_{21}\leq C\|b\|_{H^3}\sum_{i=1}^2\left(\|\Lambda_i^{\alpha} u\|_{H^3}^2+\|(\Lambda_i^\beta, \Lambda_{i''}^\beta) b_{i'}\|_{H^3}^2\right).
\]
To estimate the term $J_{22}$, we consider the following four cases
\begin{align*}
1. ~ 3=&k=j ; \quad 2. ~ 3=k\neq j;\quad 3. ~ 3=j\neq k;\quad 4. ~ j\neq 3 ~and~ k\neq3.
\end{align*}
Case 1: $3=k=j$. The terms $J_{221}$ can be bounded in same way to $K_{121}$,
\[
J_{221}\leq C\|b\|_{H^3} \left(\|(\Lambda_1^{\alpha} u_{1}, \Lambda_2^{\alpha} u_{2})\|_{H^3}^2+\|(\Lambda_{3}^\beta b_{1}, \Lambda_{3}^\beta b_{2})\|_{H^3}^2\right).
\]
\noindent Case 2: $3= k $ and $ 3\neq j$, we do the following decomposition
\begin{align*}
J_{222}&=\sum_{m=1}^2\int\partial_{3}^{3-m} b_{1} \partial_{1}\partial_{3}^{m} b_{3}\partial_{3}^{2}\partial_{1} u_{1} dx +\sum_{m=1}^2\int\partial_{3}^{3-m} b_{1} \partial_{1}\partial_{3}^{m} b_{3}\partial_{3}^{2}\partial_{2} u_{2} dx \\
&~~~+\sum_{m=1}^2\int\partial_{3}^{3-m} b_{2} \partial_{2}\partial_{3}^{m} b_{3}\partial_{3}^{2}\partial_{1} u_{1} dx +\sum_{m=1}^2\int\partial_{3}^{3-m} b_{2}\partial_{2}\partial_{3}^{m} b_{3}\partial_{3}^{2}\partial_{2} u_{2} dx \\
&:=J_{2221}+J_{2222}+J_{2223}+J_{2224}.
\end{align*}
For the first term $J_{2221}$, by using again Lemma \ref{lemmage} (ii, iii) and Young inequality, we find that
\begin{align*}
J_{2221}\leq&C\sum_{m=1}^2 \|\partial_{3}^{2}\partial_{1} u_{1}\|^{1-\frac{1}{2\alpha}} \|\Lambda_1^{\alpha}\partial_{3}^{2}\partial_{1} u_{1}\|^\frac{1}{2\alpha} \|\partial_{1}\partial_{3}^{m} b_{3}\|^{1-\frac{1}{2\beta}} \|\Lambda_{2}^{\beta}\partial_{1}\partial_{3}^{m} b_{3} \\&\times\|^\frac{1}{2\beta}\|\partial_{3}^{3-m} b_{1}\|^\frac{1}{2} \|\partial_{3}\partial_{3}^{3-m} b_{1}\|^\frac{1}{2}  \\
\leq&C \|b\|_{H^3}\|\Lambda_1^{\alpha}u_{1}\|_{H^3} \|\Lambda_1^{\beta}b_{3}\|_{H^3}^{1-\frac{1}{2\beta}} \|\Lambda_{2}^{\beta}b_{3}\|_{H^3}^{\frac{1}{2\beta}}\\
\leq &C \|b\|_{H^3}\left(\|\Lambda_1^{\alpha}u_{1}\|_{H^3}^2+ \|(\Lambda_1^\beta, \Lambda_{2}^\beta) b_{3}\|_{H^3}^2\right).
\end{align*}
In the same way, we also deduce that
\begin{align*}
J_{2222}+J_{2223}+J_{2224}\leq &C\| b\|_{H^3}\left(\| (\Lambda_1^{\alpha}, \Lambda_2^{\alpha})(u_{1}, u_{2})\|_{H^3}^2+\|(\Lambda_1^\beta, \Lambda_{2}^\beta) b_{3}\|_{H^3}^2\right).
\end{align*}
Case 3: $3=j$ and $3\neq k$. For the term $J_{223}$, we split it into two parts
\begin{align*}
J_{223}=\sum_{m=1}^2\int\partial_3^{3-m} b_{3}\partial_3 \partial_3^{m} b_{1}\partial_3^3 u_1 dx+\sum_{m=1}^2\int\partial_3^{3-m} b_{3}\partial_3 \partial_3^{m} b_{2}\partial_3^3 u_2 dx:= J_{2231}+J_{2232}.
\end{align*}
Applying again Lemma \ref{lemmage} (i, ii, iv) and Young inequality, we infer that
\begin{align*}
J_{2231}\leq&C\sum_{m=1}^2 \|\partial_{3}^3 u_{1}\|^{1-\frac{1}{2\alpha}} \|\Lambda_2^{\alpha}\partial_{3}^3 u_{1}\|^\frac{1}{2\alpha} \|\partial_{3}^{m+1} b_{1}\|^{1-\frac{1}{2\beta}} \|\Lambda_{3}^{\beta}\partial_{3}^{m+1} b_{1}\|^\frac{1}{2\beta} 
\\ &\times\|\partial_{3}^{3-m} b_{3}\|^{1-\frac{1}{2\beta}} \|\Lambda_{1}^{\beta}\partial_{3}^{3-m} b_{3}\|^\frac{1}{2\beta} \\
\leq&C \|u\|_{H^3}^{1-\frac{1}{2\alpha}} \|\Lambda_2^{\alpha} u_{1}\|_{H^3}^\frac{1}{2\alpha} \|b\|_{H^3}^{(1-\frac{1}{2\beta})(\frac{1}{2\alpha}+\frac{1}{2\beta}-1)} \|\Lambda_{3}^{\beta} b_{1}\|_{H^3}^{(1-\frac{1}{2\beta})(2-\frac{1}{2\alpha}-\frac{1}{2\beta})}\\ &\times \|\Lambda_{3}^{\beta} b_{1}\|_{H^3}^\frac{1}{2\beta} \|b\|_{H^3}^{1-\frac{1}{2\beta}} \| b\|_{H^3}^{\frac{1}{2\beta}(\frac{1}{2\alpha}+\frac{1}{2\beta}-1)} \|\Lambda_{1}^{\beta} b_{3}\|_{H^3}^{\frac{1}{2\beta}(2-\frac{1}{2\alpha}-\frac{1}{2\beta})} \\
\leq& C \|(u, b)\|_{H^3}\left(\|\Lambda_2^{\alpha} u_{1}\|_{H^3}^2+\|(\Lambda_1^\beta b_{3}, \Lambda_{3}^\beta b_{1})\|_{H^3}^2\right)
\end{align*}
and
\begin{align*}
J_{2232}\leq& C\|(u, b)\|_{H^3}\left(\|\Lambda_2^{\alpha} u_{2}\|_{H^3}^2+\|(\Lambda_1^\beta b_{3}, \Lambda_{3}^\beta b_{2})\|_{H^3}^2\right).
\end{align*}
Case 4: $ j\neq 3 $ and $k\neq3$. We write that
\begin{align*}
J_{224}&=\sum_{m=1}^2\int\partial_{3}^{3-m} b_{1}\partial_{1}\partial_{3}^{m} b_{1} \partial_{3}^3 u_{1} dx+ \sum_{m=1}^2\int \partial_{3}^{3-m} b_{1}\partial_{1}\partial_{3}^{m} b_{2} \partial_{3}^3 u_{2} dx\\
&~~~+\sum_{m=1}^2\int \partial_{3}^{3-m} b_{2}\partial_{2}\partial_{3}^{m} b_{1} \partial_{3}^3 u_{1} dx+\sum_{m=1}^2\int\partial_{3}^{3-m} b_{2}\partial_{2}\partial_{3}^{m} b_{2} \partial_{3}^3 u_{2} dx\\
&:=J_{2241}+J_{2242}+J_{2243}+J_{2244}.
\end{align*}
Thanks to Lemma \ref{lemmage} (i, ii, iv) and Young inequality, we obtain
\begin{align*}
J_{2241}\leq&C\sum_{m=1}^2\|\partial_{3}^3 u_{1}\|^{1-\frac{1}{2\alpha}} \|\Lambda_1^{\alpha}\partial_{3}^3 u_{1}\|^\frac{1}{2\alpha} \|\partial_{1}\partial_{3}^{m} b_{1}\|^{1-\frac{1}{2\beta}} \|\Lambda_{2}^{\beta}\partial_{1}\partial_{3}^{m} b_{1}\|^\frac{1}{2\beta} \\&\times \|\partial_{3}^{3-m} b_{1}\|^{1-\frac{1}{2\beta}} \|\Lambda_{3}^{\beta}\partial_{3}^{3-m} b_{1}\|^\frac{1}{2\beta} \\
\leq&C\|u\|_{H^3}^{1-\frac{1}{2\alpha}} \| \Lambda_1^{\alpha} u_{1}\|_{H^3}^\frac{1}{2\alpha} \|b\|_{H^3}^{(1-\frac{1}{2\beta})(\frac{1}{2\alpha}+\frac{1}{2\beta}-1)} \|\Lambda_{3}^{\beta} b_{1}\|_{H^3}^{(1-\frac{1}{2\beta})(2-\frac{1}{2\alpha}-\frac{1}{2\beta})} \\&\times \|\Lambda_{2}^{\beta} b_{1}\|_{H^3}^\frac{1}{2\beta} \|b\|_{H^3}^{1-\frac{1}{2\beta}} \|b\|_{H^3}^{\frac{1}{2\beta}(\frac{1}{2\alpha}+\frac{1}{2\beta}-1)} \|\Lambda_{3}^{\beta} b_{1}\|_{H^3}^{\frac{1}{2\beta}(2-\frac{1}{2\alpha}-\frac{1}{2\beta})} \\
\leq& C\|(u, b)\|_{H^3}\left(\|\Lambda_1^{\alpha} u_{1}\|_{H^3}^2+\|(\Lambda_2^\beta, \Lambda_{3}^\beta) b_{1}\|_{H^3}^2\right).
\end{align*}
Similarly,
\begin{align*}
J_{2242}+J_{2243}+J_{2244} \leq& C\|(u, b)\|_{H^3}\sum_{i=1}^2\left(\|\Lambda_{i}^{\alpha} u_{i}\|_{H^3}^2+\|(\Lambda_{i'}^\beta, \Lambda_{i''}^\beta) b_{i}\|_{H^3}^2\right).
\end{align*}
\noindent\underline{\textbf {Estimate of $H_2$.}} The term $H_2$ can be expanded under the form
\begin{align*}
H_2&=\sum_{i=1}^3 \int\partial_{i}^3 b\cdot\nabla u\cdot\partial_{i}^3 b dx+3\sum_{m=1}^2\sum_{i=1}^3 \int\partial_i^{3-m} b\cdot\nabla \partial_i^{m} u\cdot\partial_{i}^3 b dx +\sum_{i=1}^3 \int b\cdot\nabla \partial_{i}^3 u\cdot\partial_{i}^3 b dx\\&:=M_{1}+3M_{2}+M_{3}.
\end{align*}
Here we have $J_{3}+M_{3}=0$. Then we can bounded the term $M_{1}$.

\noindent\textbf {Estimate of $M_1$.} We discuss the following four cases
\begin{align*}
1. ~ i=&j=k ; \quad 2. ~ i=j\neq k;\quad 3. ~ i=k\neq j;\quad 4. ~ i\neq j\neq k.
\end{align*}
\noindent Case 1: $i=j=k$. To estimate the first item $M_{11}$, we first split it into the following four parts
\begin{align*}
M_{11}&=\sum_{i=1}^3 \int \partial_{i}^2\partial_{i'} b_{i'}\partial_{i} u_{i}\partial_{i}^2\partial_{i'} b_{i'} dx+\sum_{i=1}^3 \int \partial_{i}^2\partial_{i'} b_{i'}\partial_{i} u_{i}\partial_{i}^2\partial_{i''} b_{i''} dx \\
&~~~+\sum_{i=1}^3\int \partial_{i}^2\partial_{i''} b_{i''}\partial_{i} u_{i}\partial_{i}^2\partial_{i'} b_{i'} dx+\sum_{i=1}^3\int \partial_{i}^2\partial_{i''} b_{i''}\partial_{i} u_{i}\partial_{i}^2\partial_{i''} b_{i''} dx\\
&:=M_{111}+M_{112}+M_{113}+M_{114}.
\end{align*}
We use again Lemma \ref{lemmage} (ii, iii) and Young inequality to get
\begin{align*}
M_{111}\leq&C\sum_{i=1}^3 \|\partial_{i}^2\partial_{i'} b_{i'}\|^{1-\frac{1}{2\beta}}\|\Lambda_{i}^{\beta}\partial_{i}^2\partial_{i'} b_{i'}\|^\frac{1}{2\beta} \|\partial_{i}^2\partial_{i'} b_{i'}\|^{1-\frac{1}{2\beta}} \|\Lambda_{i''}^{\beta}\partial_{i}^2\partial_{i'} b_{i'}\|^\frac{1}{2\beta} \|\partial_i u_i\|^{\frac{1}{2}}\| \partial_{i'i} u_i\|^\frac{1}{2}\\
\leq&C\sum_{i=1}^3\|\Lambda_{i}^{\beta} b_{i'}\|_{H^3}^{2-\frac{1}{2\beta}}\| \Lambda_{i''}^{\beta} b_{i'}\|_{H^3}^\frac{1}{2\beta}\|u\|_{H^3}
\leq C\|u\|_{H^3}\sum_{i=1}^3\|(\Lambda_i^\beta, \Lambda_{i''}^\beta) b_{i'}\|_{H^3}^2.
\end{align*}
Similarly,
\begin{align*}
M_{112}+M_{113}+M_{114}\leq C\|u\|_{H^3}\sum_{i=1}^3\|(\Lambda_i^\beta, \Lambda_{i''}^\beta) b_{i'}\|_{H^3}^2.
\end{align*}
Note that the case $2, 3$ and $4$ can be treated by the way of the case 1, and then we obtain that
\begin{align*}
M_{1}\leq C\|u\|_{H^3}\sum_{i=1}^3\|(\Lambda_i^\beta, \Lambda_{i''}^\beta) b_{i'}\|_{H^3}^2.
\end{align*}

\noindent\textbf {Estimate of $M_2$.} For the last term $M_2$, we split it into two parts
\begin{align*}
M_2=\sum_{i,m=1}^2 \int \partial_i^{3-m} b\cdot\nabla \partial_i^{m} u\cdot\partial_{i}^3 b dx+\sum_{m=1}^2\int \partial_3^{3-m} b\cdot\nabla \partial_3^{m} u\cdot\partial_3^2 b dx:=M_{21}+M_{22}.
\end{align*}
The way of the first term $M_{21}$ is similar to the term $K_{11}$. We obtain
\begin{align*}\label{k11}
M_{21}\leq C\|b\|_{H^3}\sum_{i=1}^2\left(\|\Lambda_i^{\alpha} u\|_{H^3}^2+\|(\Lambda_i^\beta, \Lambda_{i''}^\beta) b_{i'}\|_{H^3}^2\right).
\end{align*}
To estimate the second term $M_{22}$, we will discuss it in three cases
\begin{align*}
1. ~ 3=&k ; \qquad 2. ~ 3=j\neq k;\qquad 3. ~ 3\neq j\neq k.
\end{align*}
Case 1: $3=k$. We decompose $M_{221}$ into four terms also by using the divergence free condition $div~u=div~b=0$,
\begin{align*}
M_{221}&=\sum_{m=1}^2\int \partial_3^{3-m} b\cdot\nabla\partial_3^{m-1}\partial_{1} u_{1}\partial_{3}^2\partial_{1} b_{1}dx +\sum_{m=1}^2\int\partial_3^{3-m} b\cdot\nabla\partial_3^{m-1} \partial_{1} u_{1}\partial_{3}^2\partial_{2} b_{2}dx \\
&\quad+\sum_{m=1}^2\int \partial_3^{3-m} b\cdot\nabla\partial_3^{m-1}\partial_{2} u_{2}\partial_{3}^2\partial_{1} b_{1}dx +\sum_{m=1}^2\int\partial_3^{3-m} b\cdot\nabla\partial_3^{m-1} \partial_{2} u_{2}\partial_{3}^2\partial_{2} b_{2}dx \\
&:=M_{2211}+M_{2212}+M_{2213}+M_{2214}.
\end{align*}
We start with the first term $M_{2211}$. The remaining three terms are treated similarly to the term $M_{2211}$. By using Lemma \ref{lemmage} (ii, iii) and Young inequality, we obtain
\begin{align*}
M_{2211}\leq&C\sum_{m=1}^2 \|\nabla\partial_3^{m-1}\partial_{1} u_{1}\|^{1-\frac{1}{2\alpha}} \|\Lambda_{1}^{\alpha}\nabla \partial_3^{m-1}\partial_{1} u_{1}\|^\frac{1}{2\alpha} \|\partial_{3}^2\partial_{1} b_{1}\|^{1-\frac{1}{2\beta}} \|\Lambda_{3}^{\beta}\partial_{3}^2\partial_{1} b_{1}\|^\frac{1}{2\beta} \\&\times\|\partial_3^{3-m} b\|^{\frac{1}{2}} \|\partial_{2}\partial_3^{3-m} b\|^\frac{1}{2}\\
\leq&C\|\Lambda_{1}^{\alpha} u_{1}\|_{H^3}\|\Lambda_{3}^{\beta} b_{1}\|_{H^3}\|b\|_{H^3}
\leq C\|b\|_{H^3}(\|\Lambda_{1}^{\alpha} u_{1}\|_{H^3}^2+\|\Lambda_{3}^{\beta} b_{1}\|_{H^3}^2)
\end{align*}
and
\begin{align*}
M_{2212}+M_{2213}+M_{2214}\leq C\|b\|_{H^3}\left(\|(\Lambda_{1}^{\alpha} u_{1}, \Lambda_{2}^{\alpha} u_{2})\|_{H^3}^2+\|(\Lambda_{3}^{\beta} b_{1}, \Lambda_{3}^{\beta} b_{2})\|_{H^3}^2\right).
\end{align*}
Case 2: $3=j\neq k$. We expand
\begin{align*}
M_{222}=\sum_{m=1}^2 \int\partial_3^{3-m} b_{3}\partial_{3}\partial_{3}^{m} u_{1}\partial_{3}^3 b_{1}dx +\sum_{m=1}^2 \int\partial_3^{3-m} b_{3}\partial_{3}\partial_{3}^{m} u_{2}\partial_{3}^3 b_{2} dx :=M_{2221}+M_{2222}.
\end{align*}
As we have done to handle the first term $M_{2221}$, we write thanks to Lemma \ref{lemmage} (i, ii, iv) and Young inequality that
\begin{align*}
M_{2221}
\leq&C\|\partial_{3}^{m+1} u_{1}\|^{1-\frac{1}{2\alpha}} \|\Lambda_1^{\alpha}\partial_{3}^{m+1} u_{1}\|^\frac{1}{2\alpha} \|\partial_{3}^3 b_{1}\|^{1-\frac{1}{2\beta}} \|\Lambda_{3}^{\beta}\partial_{3}^3 b_{1}\|^\frac{1}{2\beta} \|\partial_{3}^{3-m} b_{3}\|^{1-\frac{1}{2\beta}} \|\Lambda_{2}^{\beta}\partial_{3}^{3-m} b_{3}\|^\frac{1}{2\beta} \\
\leq&C\|u\|_{H^3}^{1-\frac{1}{2\alpha}} \|\Lambda_1^{\alpha} u_{1}\|_{H^3}^\frac{1}{2\alpha} \|b\|_{H^3}^{(1-\frac{1}{2\beta})(\frac{1}{2\alpha}+\frac{1}{2\beta}-1)} \|\Lambda_{3}^{\beta} b_{1}\|_{H^3}^{(1-\frac{1}{2\beta})(2-\frac{1}{2\alpha}-\frac{1}{2\beta})} \\&\times \|\Lambda_{3}^{\beta} b_{1}\|_{H^3}^\frac{1}{2\beta} \|b\|_{H^3}^{1-\frac{1}{2\beta}} \|b\|_{H^3}^{\frac{1}{2\beta}(\frac{1}{2\alpha}+\frac{1}{2\beta}-1)} \|\Lambda_{2}^{\beta} b_{3}\|_{H^3}^{\frac{1}{2\beta}(2-\frac{1}{2\alpha}-\frac{1}{2\beta})} \\
\leq& C\|(u, b)\|_{H^3}\left(\|\Lambda_1^{\alpha} u_{1}\|_{H^3}^2+\|(\Lambda_2^\beta b_{3}, \Lambda_{3}^\beta b_{1})\|_{H^3}^2\right).
\end{align*}
In a similar way, this yields
\begin{align*}
M_{2222}\leq& C\|(u, b)\|_{H^3}\left(\|\Lambda_1^{\alpha} u_{2}\|_{H^3}^2+\|(\Lambda_2^\beta b_{3}, \Lambda_{3}^\beta b_{2})\|_{H^3}^2\right).
\end{align*}
Case 3: $3\neq j\neq k$. We divided $M_{223}$ into four terms as follows
\begin{align*}
M_{223}&=\sum_{m=1}^2\int\partial_3^{3-m} b_{1}\partial_{1}\partial_3^{m} u_{1}\partial_{3}^3 b_{1}dx +\sum_{m=1}^2\int \partial_3^{3-m} b_{1}\partial_{1}\partial_3^{m} u_{2}\partial_{3}^3 b_{2} dx \\
&~~~+\sum_{m=1}^2\int \partial_3^{3-m} b_{2}\partial_{2}\partial_3^{m} u_{1}\partial_{3}^3 b_{1} dx +\sum_{m=1}^2\int \partial_3^{3-m} b_{2}\partial_{2}\partial_3^{m} u_{2}\partial_{3}^3 b_{2} dx\\
&:=M_{2231}+M_{2232}+M_{2233}+M_{2234}.
\end{align*}
The terms $M_{2231}$ and $M_{2211}$ can be controlled in same way and hence, we find
\begin{align*}
M_{2231}\leq&C\sum_{m=1}^2 \|\partial_{1}\partial_3^{m} u_{1}\|^{1-\frac{1}{2\alpha}} \|\Lambda_{1}^{\alpha} \partial_{1}\partial_3^{m} u_{1}\|^\frac{1}{2\alpha} \|\partial_{3}^3 b_{1}\|^{1-\frac{1}{2\beta}} \|\Lambda_{3}^{\beta}\partial_{3}^3 b_{1}\|^\frac{1}{2\beta} \|\partial_3^{3-m} b_{1}\|^{\frac{1}{2}} \|\partial_{2}\partial_{3}^{3-m} b_{1}\|^\frac{1}{2}\\
\leq&C \|\Lambda_{1}^{\alpha} u_{1}\|_{H^3} \|\Lambda_{3}^{\beta} b_{1}\|_{H^3}\|b\|_{H^3}
\leq C \|b\|_{H^3}(\|\Lambda_{1}^{\alpha} u_{1}\|_{H^3}^2+\|\Lambda_{3}^{\beta} b_{1}\|_{H^3}^2)
\end{align*}
and
\begin{align*}
M_{2232}+M_{2233}+M_{2234}\leq C\|b\|_{H^3}\|(\Lambda_{2}^{\alpha} u_{1}, \Lambda_{1}^{\alpha}u_{2}, \Lambda_{2}^{\alpha} u_{2}, \Lambda_{3}^{\beta} b_{1}, \Lambda_{3}^{\beta} b_{2})\|_{H^3}^2.
\end{align*}
Combining these estimates of $G_1, H_1, G_2$ and $H_2$, we finally get that
\begin{align*}
\frac{1}{2}\frac{d}{d t}&\|(u, b)\|_{H^3}^2 +\|(\nu_1^{\frac{1}{2}}\Lambda_1^{\alpha}, \nu_2^{\frac{1}{2}}\Lambda_2^{\alpha})u\|_{H^3}^2+\sigma\nu_3 \|\Lambda_3^{\alpha}u\|_{H^3}^2+\mu\sum_{i=1}^3\|(\Lambda_{i'}^\beta, \Lambda_{i''}^\beta) b_{i}\|_{H^3}^2\\
&\leq C\|(u, b)\|_{H^3}\left(\| (\Lambda_1^{\alpha}, \Lambda_2^{\alpha}) u\|_{H^3}^2+\sigma\nu_3 \|\Lambda_3^{\alpha}u\|_{H^3}^2+\sum_{i=1}^3\|(\Lambda_{i'}^\beta, \Lambda_{i''}^\beta) b_{i}\|_{H^3}^2\right).
\end{align*}
Integrating in time we find
\begin{align*}
\sup_{\tau\in [0, t]}&\|(u, b)\|_{H^3}^2 +\int_0^t \|(\nu_1^{\frac{1}{2}}\Lambda_1^{\alpha}, \nu_2^{\frac{1}{2}}\Lambda_2^{\alpha})u\|_{H^3}^2 +\sigma\nu_3 \|\Lambda_3^{\alpha}u\|_{H^3}^2 +\mu\sum_{i=1}^3\|(\Lambda_{i'}^\beta, \Lambda_{i''}^\beta) b_{i}\|_{H^3}^2d\tau\\
&\leq C\|(u_{in}, b_{in})\|_{H^3}^2+C\sup_{\tau\in [0, t]}\|(u, b)\|_{H^3}
\\&\hspace{3.2cm} \times \left(\int_0^t\| (\Lambda_1^{\alpha}, \Lambda_2^{\alpha}) u\|_{H^3}^2+\sigma\nu_3 \|\Lambda_3^{\alpha}u\|_{H^3}^2+\sum_{i=1}^3\|(\Lambda_{i'}^\beta, \Lambda_{i''}^\beta) b_{i}\|_{H^3}^2d\tau\right).
\end{align*}
We deduce by the definition (\ref{energyfunctional}) that
\[
E(t)\leq E(0)+E(t)^{\frac{3}{2}}, \quad \forall ~ t\geq 0.
\]
This end the proof for the global stability of the system (\ref{bhe}). The uniqueness will be proven in the following. We firstly assume that $(u^{(1)}, \pi^{(1)}, b^{(1)})$ and $(u^{(2)}, \pi^{(2)}, b^{(2)})$ are two pairs of solutions of the system (\ref{bhe}) with the same initial data $(u_{in}, b_{in})$ on $[0, t]$. Denote
$$ \bar{u}=u^{(1)}-u^{(2)}, ~\bar{\pi}=\pi^{(1)}-\pi^{(2)}, ~\bar{b}=b^{(1)}-b^{(2)}.$$
Each equation corresponds to a difference, this yields
\begin{eqnarray}\label{bhed}
\begin{cases}
\partial_{t} \bar{u}+ u^{(1)}\cdot\nabla\bar{u} + \bar{u} \cdot\nabla u^{(2)}=-( \nu_1\Lambda_1^{2\alpha} + \nu_2\Lambda_2^{2\alpha}+ \sigma\nu_3\Lambda_2^{2\alpha})~ \bar{u}+\nabla \bar{\pi}-b^{(1)} \cdot\nabla \bar{b} \\ \hspace{9.5cm} -\bar{b} \cdot\nabla b^{(2)} +\partial_3 \bar{b}, \\
\partial_{t} \bar{b}+ u^{(1)} \cdot\nabla \bar{b}+\bar{u} \cdot\nabla b^{(2)}~=-\mu\begin{bmatrix} \Lambda_2^{2\beta} + \Lambda_3^{2\beta}\\ \Lambda_1^{2\beta} + \Lambda_3^{2\beta}\\ \Lambda_1^{2\beta} + \Lambda_2^{2\beta} \end{bmatrix} \bar{b} -b^{(1)}\cdot\nabla \bar{u}-\bar{b}\cdot\nabla u^{(2)} +\partial_3 \bar{u}, \\
div~ \bar{u}=div~ \bar{b}=0, \\
(\bar{u}, \bar{b})\vert _{t=0}=(\bar{u}_{in}, \bar{b}_{in}). \\
\end{cases}
\end{eqnarray}
Taking the $L^2$ inner product of the system (\ref{bhed}) with $(\bar{u}, \bar{b})$, we get that
\begin{align*}
\frac{1}{2}\frac{d}{dt}&\|(\bar{u}, \bar{b})\|^2+\|(\nu_1^{\frac{1}{2}}\Lambda_{1}^{\alpha}, \nu_2^{\frac{1}{2}}\Lambda_{2}^{\alpha}) \bar{u}\|^{2}\\
&\hspace{3cm}+\sigma\nu_3\|\Lambda_3^{\alpha}\bar{u}_3\|^2+\mu\left(\|(\Lambda_2^{\beta}, \Lambda_3^{\beta})\bar{b}_1\|^2+\|(\Lambda_1^{\beta}, \Lambda_3^{\beta})\bar{b}_2\|^2
+\|(\Lambda_1^{\beta}, \Lambda_2^{\beta})\bar{b}_3\|^2\right)\\
&=-\int \bar{u}\cdot\nabla u^{(2)}\cdot\bar{u} dx-\int \bar{b}\cdot\nabla b^{(2)}\cdot\bar{u} dx-\int \bar{u}\cdot\nabla b^{(2)}\cdot\bar{b} dx-\int \bar{b}\cdot\nabla u^{(2)}\cdot\bar{b} dx  \\&:=L_{1}+L_{2}+L_{3}+L_{4}.
\end{align*}
Form Lemma \ref{lemmage} (iii) and Young inequality, we infer
\begin{align*}
L_{1}\leq &C\|\bar{u}\|^{2-\frac{1}{\alpha}}\|\Lambda_{1}^{\alpha} \bar{u}\|^\frac{1}{2\alpha}\|\Lambda_{2}^{\alpha} \bar{u}\|^\frac{1}{2\alpha}\|\nabla u^{(2)}\|^{\frac{1}{2}}\| \partial_{3}\nabla u^{(2)}\|^\frac{1}{2}
\leq C\|\bar{u}\|^{2}+\frac{1}{6}\|(\nu_1^{\frac{1}{2}}\Lambda_{1}^{\alpha}, \nu_2^{\frac{1}{2}}\Lambda_{2}^{\alpha}) \bar{u}\|^{2}.
\end{align*}
To estimate the term $L_{2}$, it can be rewritten under the form
\begin{align*}
L_{2}= -\sum_{i=1}^2 \int \bar{b}_i\partial_i b^{(2)}\cdot\bar{u} dx- \int \bar{b}_3\partial_3 b^{(2)}\cdot\bar{u} dx:=L_{21}+L_{22}.
\end{align*}
For the term $L_{21}$, this yields according to again Lemma \ref{lemmage} (iii) and Young inequality, we find that
\begin{align*}
L_{21}\leq& C\sum_{i=1}^2\|\bar{u}\|^{1-\frac{1}{2\alpha}}\|\Lambda_{i}^{\alpha} \bar{u}\|^\frac{1}{2\alpha}\|\bar{b}_i\|^{1-\frac{1}{2\beta}}\|\Lambda_{i'}^{\beta} \bar{b}_i\|^\frac{1}{2\beta}\|\partial_i b^{(2)}\|^{\frac{1}{2}}\| \partial_{i''i} b^{(2)}\|^\frac{1}{2}\\
\leq& \sum_{i=1}^2 \left( C\|(\bar{u}, \bar{b}_{i})\|^{2} + \frac{\mu}{12}\|\Lambda_{i'}^{\beta} \bar{b}_{i}\|^{2}\right)+\frac{1}{12}\|(\nu_1^{\frac{1}{2}}\Lambda_{1}^{\alpha}, \nu_2^{\frac{1}{2}}\Lambda_{2}^{\alpha}) \bar{u}\|^{2}
\end{align*}
and
\begin{align*}
L_{22}\leq& C\|\bar{u}\|^{1-\frac{1}{2\alpha}}\|\Lambda_{1}^{\alpha} \bar{u}\|^\frac{1}{2\alpha}\|\bar{b}_3\|^{1-\frac{1}{2\beta}}\|\Lambda_{2}^{\beta} \bar{b}_3\|^\frac{1}{2\beta}\|\partial_3 b^{(2)}\|^{\frac{1}{2}}\| \partial_{3}^3 b^{(2)}\|^\frac{1}{2}\\
\leq& C\|(\bar{u}, \bar{b}_{3})\|^{2} +\frac{1}{12}\|\nu_1^{\frac{1}{2}}\Lambda_{1}^{\alpha} \bar{u}\|^{2}+ \frac{\mu}{12}\|\Lambda_{2}^{\beta} \bar{b}_{3}\|^{2}.
\end{align*}
By the same computation as above, we obtain
\begin{align*}
L_{3}\leq& C\|(\bar{u}, \bar{b})\|^{2} + \frac{1}{6}\|(\nu_1^{\frac{1}{2}}\Lambda_{1}^{\alpha}, \nu_2^{\frac{1}{2}}\Lambda_{2}^{\alpha}) \bar{u}\|^{2}+ \frac{\mu}{6}\sum_{i=1}^3\|(\Lambda_{i'}^{\beta}, \Lambda_{i''}^{\beta})\bar{b}_i\|^{2}.
\end{align*}
To bounded the term $L_{4}$, we decompose it into three terms as follow
\begin{align*}
L_{4}&=-\sum_{i=1}^3 \int \bar{b}_i\partial_i u_i^{(2)}\bar{b}_i dx- \sum_{i=1}^3 \int \bar{b}_i\partial_i u_{i'}^{(2)}\bar{b}_{i'} dx- \sum_{i=1}^3 \int \bar{b}_i\partial_i u_{i''}^{(2)}\bar{b}_{i''} dx\\
 &:=L_{41}+L_{42}+L_{43}.
\end{align*}
For $L_{41}$, we get by using again Lemma \ref{lemmage} (iii) and Young inequality that
\begin{align*}
L_{41}\leq& C\sum_{i=1}^3\|\bar{b}_i\|^{2-\frac{1}{\beta}} \|\Lambda_{i'}^{\beta} \bar{b}_i\|^\frac{1}{2\beta} \|\Lambda_{i''}^{\beta} \bar{b}_i\|^\frac{1}{2\beta} \|\partial_i u_i^{(2)}\|^{\frac{1}{2}} \|\partial_{i}^3 u_i^{(2)}\|^\frac{1}{2} \leq C \|\bar{b}\|^{2} + \frac{\mu}{18}\sum_{i=1}^3\|(\Lambda_{i'}^{\beta}, \Lambda_{i''}^{\beta})\bar{b}_i\|^{2}.
\end{align*}
In a similar way, we thus obtain
\begin{align*}
L_{42}+L_{43}\leq C\|\bar{b}\|^{2} + \frac{\mu}{9}\sum_{i=1}^3\|(\Lambda_{i'}^{\beta}, \Lambda_{i''}^{\beta})\bar{b}_i\|^{2}.
\end{align*}
Combining this estimate, we have
\begin{align*}
\frac{1}{2}&\frac{d}{dt}\|(\bar{u}, \bar{b})\|^2+\|(\nu_1^{\frac{1}{2}}\Lambda_{1}^{\alpha}, \nu_2^{\frac{1}{2}}\Lambda_{2}^{\alpha}) \bar{u}\|^{2}\\
&\hspace{3.3cm}+\sigma\nu_3\|\Lambda_3^{\alpha}\bar{u}_{3}\|^2+\mu(\|(\Lambda_2^{\beta}, \Lambda_3^{\beta})\bar{b}_1\|^2+\|(\Lambda_1^{\beta}, \Lambda_3^{\beta})\bar{b}_2\|^2
+\|(\Lambda_1^{\beta}, \Lambda_2^{\beta})\bar{b}_3\|^2)\\
&\leq C\|(\bar{u}, \bar{b})\|^{2} + \frac{1}{2}\|(\nu_1^{\frac{1}{2}}\Lambda_{1}^{\alpha}, \nu_2^{\frac{1}{2}}\Lambda_{2}^{\alpha}) \bar{u}\|^{2}+ \frac{\mu}{2}\left(\|(\Lambda_2^{\beta}, \Lambda_3^{\beta})\bar{b}_1\|^2+\|(\Lambda_1^{\beta}, \Lambda_3^{\beta})\bar{b}_2\|^2
+\|(\Lambda_1^{\beta}, \Lambda_2^{\beta})\bar{b}_3\|^2\right).
\end{align*}
It follows form Gr\"{o}nwall inequality that
\begin{align*}
\|(\bar{u}, \bar{b})\|^2=0.
\end{align*}
This finishes the proof of Theorem \ref{theoreml}.
\end{proof}

\section{Proof of Theorem \ref{theorem2}.}\label{sec4}
On the basis of Theorem \ref{theoreml}, we obtain that the solution $(u^{\nu}, b^{\nu})$ of the system (\ref{bhev}) satisfies
\begin{align}\label{neng21}
\|(u^{\nu}, b^{\nu})\|_{H^3}^2+&\int_0^t \|(\Lambda_1^{\alpha}, \Lambda_2^{\alpha})u^{\nu}\|_{H^3}^2 d\tau+\nu\int_0^t \|\Lambda_3^{\alpha}u^{\nu}\|_{H^3}^2 d\tau  \nonumber \\ & +\int_0^t \|(\Lambda_2^{\beta}, \Lambda_3^{\beta})b^{\nu}_1\|_{H^3}^2+\|(\Lambda_1^{\beta}, \Lambda_3^{\beta})b^{\nu}_2\|_{H^3}^2+\|(\Lambda_1^{\beta}, \Lambda_2^{\beta})b^{\nu}_3\|_{H^3}^2 d\tau\leq C\epsilon^2
\end{align}
and the solution $(u^{0}, b^{0})$ of the system (\ref{bhev0}) satisfies
\begin{align}\label{neng22}
\|(u^{0}, b^{0})\|_{H^3}^2+&\int_0^t \|(\Lambda_1^{\alpha}, \Lambda_2^{\alpha})u^{0}\|_{H^3}^2 d\tau    \nonumber \\
&+\int_0^t \|(\Lambda_2^{\beta}, \Lambda_3^{\beta})b^{0}_1\|_{H^3}^2+\|(\Lambda_1^{\beta}, \Lambda_3^{\beta})b^{0}_2\|_{H^3}^2+\|(\Lambda_1^{\beta}, \Lambda_2^{\beta})b^{0}_3\|_{H^3}^2 d\tau\leq C\epsilon^2.
\end{align}
We now show that the convergence rate of the viscosity solutions of the system (\ref{bhev}) to the system (\ref{bhev0}) in the section.

\begin{proof}
Suppose that $(u^{\nu}, b^{\nu})$ and $(u^{0}, b^{0})$ are the solution to the system (\ref{bhev}) and the system (\ref{bhev0}) with the same initial data $(u_{0}, b_{0})$, respectively. And denote
$$\tilde{u}:=u^{\nu}-u^{0}, \quad \tilde{p}:=p^{\nu}-p^{0}, \quad \tilde{b}:=b^{\nu}-b^{0}$$
satisfy
\begin{eqnarray}\label{bhevd}
\begin{cases}
\partial_{t} \tilde{u}+ u^{\nu}\cdot\nabla\tilde{u}=-\left( \Lambda_1^{2\alpha} + \Lambda_2^{2\alpha} \right) \tilde{u} - \nu\Lambda_3^{2\alpha}u^{\nu} +\nabla \tilde{p} -\tilde{u}\cdot\nabla u^0-b^{\nu} \cdot\nabla \tilde{b}-\tilde{b} \cdot\nabla b^{0}+\partial_3 \tilde{b}, \\
\partial_{t} \tilde{b}+ u^{\nu} \cdot\nabla \tilde{b}=-\begin{bmatrix} \Lambda_2^{2\beta} + \Lambda_3^{2\beta}\\ \Lambda_1^{2\beta} + \Lambda_3^{2\beta}\\ \Lambda_1^{2\beta} + \Lambda_2^{2\beta} \end{bmatrix} \tilde{b}-\tilde{u}\cdot\nabla b^0-b^{\nu} \cdot\nabla \tilde{u}-\tilde{b} \cdot\nabla u^{0}+\partial_3 \tilde{u}, \\
div~ \tilde{u}=div~ \tilde{b}=0, \\
(\tilde{u}, \tilde{b})\vert _{t=0}=(0, 0). \\
\end{cases}
\end{eqnarray}
Dotting (\ref{bhevd}) by $(\tilde{u}, \tilde{b})$ in $L^2$. And applying $\partial_i$ to (\ref{bhevd}) and taking $L^2$-scalar product with $(\partial_i\tilde{u}, \partial_i\tilde{b})$. Gathering this estimate, we obtain that
\begin{align*}
\frac{1}{2}\frac{d}{dt}\|(\tilde{u}, \tilde{b})\|_{H^1}^2+&\|(\Lambda_{1}^{\alpha}, \Lambda_{2}^{\alpha}) \tilde{u}\|^{2} +\|(\Lambda_2^{\beta}, \Lambda_3^{\beta})\tilde{b}_1\|_{H^1}^2\\ &+\|(\Lambda_1^{\beta}, \Lambda_3^{\beta})\tilde{b}_2\|_{H^1}^2
+\|(\Lambda_1^{\beta}, \Lambda_2^{\beta})\tilde{b}_3\|_{H^1}^2:=\sum_{i=1}^5 N_{i}+\sum_{i=1}^5 Q_{i}+\sum_{i=1}^4 R_{i},
\end{align*}
where
\begin{align*}
&N_{1}=-\nu\int \Lambda_3^{2\alpha}u^{\nu}\cdot\tilde{u}dx, &&N_{2}=-\int \tilde{u}\cdot\nabla u^{0}\cdot\tilde{u} dx,\\
&N_{3}=-\int \tilde{b}\cdot\nabla b^{0}\cdot\tilde{u} dx,  &&N_{4}=-\int \tilde{u}\cdot\nabla b^{0}\cdot\tilde{b} dx, \\
&N_{5}=-\int \tilde{b}\cdot\nabla u^{0}\cdot\tilde{b} dx,   &&Q_{1}=-\nu\sum_{i=1}^3\int \Lambda_3^{2\alpha}\partial_{i} u^{\nu}\cdot\partial_{i}\tilde{u}dx,\\
&Q_{2}=-\sum_{i=1}^3\int \partial_{i}u^{\nu}\cdot\nabla\tilde{u}\cdot\partial_{i}\tilde{u}dx,    &&Q_{3}=-\sum_{i=1}^3\int \partial_{i}(\tilde{u}\cdot\nabla u^0)\cdot\partial_{i}\tilde{u}dx,\\
&Q_{4}=-\sum_{i=1}^3\int \partial_{i}(b^{\nu}\cdot\nabla\tilde{b})\cdot\partial_{i}\tilde{u}dx, &&Q_{5}=-\sum_{i=1}^3\int \partial_{i}(\tilde{b} \cdot\nabla b^{0})\cdot\partial_{i}\tilde{u}dx, \\
&R_{1}=-\sum_{i=1}^3\int \partial_{i}(\tilde{u}\cdot\nabla b^{0})\cdot\partial_{i}\tilde{b}dx,   &&R_{2}=-\sum_{i=1}^3\int \partial_{i}(b^{\nu}\cdot\nabla\tilde{u})\cdot\partial_{i}\tilde{b}dx, \\
&R_{3}=-\sum_{i=1}^3\int \partial_{i}u^{\nu}\cdot\nabla\tilde{b}\cdot\partial_{i}\tilde{b}dx,   &&R_{4}=-\sum_{i=1}^3\int \partial_{i}(\tilde{b}\cdot\nabla u^0)\cdot\partial_{i}\tilde{b}dx. &
\end{align*}
Combining the H\"{o}lder inequality, Lemma \ref{lemmage} (i) and the inequality (\ref{neng21}), we find that
\begin{align*}
N_{1}\leq \frac{\nu}{2}\|\Lambda_3^{2\alpha}u^{\nu}\|^2+\frac{\nu}{2}\|\tilde{u}\|^2\leq C\frac{\nu}{2}\|u^{\nu}\|_{H^2}^2+\frac{\nu}{2}\|\tilde{u}\|^2\leq C\nu+\frac{\nu}{2}\|\tilde{u}\|^2
\end{align*}
and
\begin{align*}
Q_{1}\leq C\nu+\frac{\nu}{2}\|\tilde{u}\|_{H^1}^2.
\end{align*}
To estimate the term $N_{2}$ we use Lemma \ref{lemmage} (iii) and Young inequality to get
\begin{align*}
N_{2}\leq &C\|\tilde{u}\|^{2-\frac{1}{\alpha}}\|\Lambda_{1}^{\alpha} \tilde{u}\|^\frac{1}{2\alpha}\|\Lambda_{2}^{\alpha} \tilde{u}\|^\frac{1}{2\alpha}\|\nabla u^{0}\|^{\frac{1}{2}}\| \partial_{3}\nabla u^{0}\|^\frac{1}{2}
\leq C\|\tilde{u}\|^{2}+\frac{1}{18}\|(\Lambda_{1}^{\alpha}, \Lambda_{2}^{\alpha}) \tilde{u}\|^{2}.
\end{align*}
For the term $N_{3}$, we rewrite
\begin{align*}
N_{3}= -\sum_{i=1}^2 \int \tilde{b}_i\partial_i b^{0}\cdot \tilde{u} dx- \int \tilde{b}_3\partial_3 b^{0}\cdot \tilde{u} dx:=N_{31}+N_{32}.
\end{align*}
This yields by using again Lemma \ref{lemmage} (iii) and Young inequality that
\begin{align*}
N_{31}\leq &C\sum_{i=1}^2\|\tilde{u}\|^{1-\frac{1}{2\alpha}}\|\Lambda_{1}^{\alpha} \tilde{u}\|^\frac{1}{2\alpha}\|\tilde{b}_i\|^{1-\frac{1}{2\beta}}\|\Lambda_{3}^{\beta} \tilde{b}_i\|^\frac{1}{2\beta}\|\partial_i b^{0}\|^{\frac{1}{2}}\| \partial_{2i} b^{0}\|^\frac{1}{2}\\
\leq &C\|(\tilde{u}, \tilde{b})\|^{2}+\frac{1}{36}\|\Lambda_{1}^{\alpha} \tilde{u}\|^{2}+\frac{1}{36}\sum_{i=1}^2\|\Lambda_{3}^{\beta} \tilde{b}_i\|^2
\end{align*}
and
\[
N_{32}\leq C\|(\tilde{u}, \tilde{b})\|^{2}+\frac{1}{36}\|\Lambda_{1}^{\alpha} \tilde{u}\|^{2}+\frac{1}{36}\|\Lambda_{2}^{\beta} \tilde{b}_3\|^2.
\]
The term $N_{4}$ can be estimated in the same way as the term $N_{3}$. For the last term $N_{5}$, we can decompose that
\begin{align*}
N_{5}= -\sum_{i=1}^3 \int \tilde{b}_i\partial_i u^{0}_{i} \tilde{b}_{i} dx-\sum_{i=1}^3 \int \tilde{b}_i\partial_i u^{0}_{i'} \tilde{b}_{i'} dx-\sum_{i=1}^3 \int \tilde{b}_i\partial_i u^{0}_{i'} \tilde{b}_{i'} dx:=N_{51}+N_{52}+N_{53}.
\end{align*}
The first term $N_{51}$ can be estimated by using again Lemma \ref{lemmage} (iii) and Young inequality,
\begin{align*}
N_{51}\leq &C\sum_{i=1}^3\|\tilde{b}_i\|^{2-\frac{1}{\beta}}\|\Lambda_{i'}^{\beta} \tilde{b}_i\|^\frac{1}{2\beta}\|\Lambda_{i''}^{\beta} \tilde{b}_i\|^\frac{1}{2\beta}\|\partial_i u^{0}_i\|^{\frac{1}{2}}\| \partial_{i}^{2} u^{0}_i\|^\frac{1}{2} \leq C\|\tilde{b}\|^{2}+\frac{1}{54}\sum_{i=1}^3\|(\Lambda_{i'}^{\beta}, \Lambda_{i''}^{\beta}) \tilde{b}_i\|^{2}.
\end{align*}
In the same way,
\begin{align*}
N_{52}+N_{53}\leq &C\|\tilde{b}\|^{2}+\frac{1}{27}\sum_{i=1}^3\left(\|\Lambda_{i'}^{\beta} \tilde{b}_i\|^{2}+\|\Lambda_{i''}^{\beta} \tilde{b}_{i'}\|^{2}+\|\Lambda_{i}^{\beta} \tilde{b}_{i''}\|^{2}\right).
\end{align*}
For the term $Q_{2}$ we use again Lemma \ref{lemmage} (iii) and Young inequality to obtain
\begin{align*}
Q_{2}& \leq C\sum_{i=1}^3 \|\partial_{i}\tilde{u}\|^{1-\frac{1}{2\alpha}} \|\Lambda_{1}^{\alpha} \partial_{i}\tilde{u}\|^\frac{1}{2\alpha} \|\nabla\tilde{u}\|^{1-\frac{1}{2\alpha}} \|\Lambda_{2}^{\alpha}\nabla\tilde{u}\|^\frac{1}{2\alpha}\|\partial_{i} u^{\nu}\|^{\frac{1}{2}}\|\partial_{3i} u^{\nu}\|^\frac{1}{2}\\&
\leq C\|\tilde{u}\|_{H^1}^{2}+\frac{1}{18}\|(\Lambda_{1}^{\alpha}, \Lambda_{2}^{\alpha}) \tilde{u}\|_{H^1}^{2}.
\end{align*}
In the similar way,
\begin{align*}
Q_{3}\leq C\|\tilde{u}\|_{H^1}^{2}+\frac{1}{18}\|(\Lambda_{1}^{\alpha}, \Lambda_{2}^{\alpha}) \tilde{u}\|_{H^1}^{2}.
\end{align*}
For the term $Q_{4}$ and $R_{2}$, we find that
\begin{align*}
Q_{4}+R_{2}= -\sum_{i=1}^3 \int \partial_{i}b^{\nu}\cdot\nabla\tilde{b}\cdot\partial_{i}\tilde{u}dx -\sum_{i=1}^3 \int \partial_{i}b^{\nu}\cdot\nabla\tilde{u}\cdot\partial_{i}\tilde{b}dx:=Q_{41}+R_{21}.
\end{align*}
The term $Q_{41}$ can be bounded by using Lemma \ref{lemmage} (iii) and Young inequality
\begin{align*}
Q_{41}=&-\sum_{i=1}^3 \int \partial_{i}b^{\nu}\cdot\nabla\tilde{b}_{1}\cdot\partial_{i}\tilde{u}_{1}dx-\sum_{i=1}^3 \int \partial_{i}b^{\nu}\cdot\nabla\tilde{b}_{2}\cdot\partial_{i}\tilde{u}_{2}dx-\sum_{i=1}^3 \int \partial_{i}b^{\nu}\cdot\nabla\tilde{b}_{3}\cdot\partial_{i}\tilde{u}_{3}dx\\
\leq &C\sum_{i=1}^3 \|\partial_{i}\tilde{u}_{1}\|^{1-\frac{1}{2\alpha}} \|\Lambda_{1}^{\alpha} \partial_{i}\tilde{u}_{1}\|^\frac{1}{2\alpha} \|\nabla\tilde{b}_{1}\|^{1-\frac{1}{2\beta}} \|\Lambda_{2}^{\beta}\nabla\tilde{b}_{1}\|^\frac{1}{2\beta} \|\partial_{i} b^{\nu}\|^{\frac{1}{2}} \|\partial_{3i} b^{\nu}\|^\frac{1}{2} \\
&+C\sum_{i=1}^3 \|\partial_{i}\tilde{u}_{2}\|^{1-\frac{1}{2\alpha}} \|\Lambda_{2}^{\alpha} \partial_{i}\tilde{u}_{2}\|^\frac{1}{2\alpha} \|\nabla\tilde{b}_{2}\|^{1-\frac{1}{2\beta}} \|\Lambda_{1}^{\beta}\nabla\tilde{b}_{2}\|^\frac{1}{2\beta} \|\partial_{i} b^{\nu}\|^{\frac{1}{2}} \|\partial_{3i} b^{\nu}\|^\frac{1}{2} \\
&+C\sum_{i=1}^3 \|\partial_{i}\tilde{u}_{3}\|^{1-\frac{1}{2\alpha}} \|\Lambda_{1}^{\alpha} \partial_{i}\tilde{u}_{3}\|^\frac{1}{2\alpha} \|\nabla\tilde{b}_{3}\|^{1-\frac{1}{2\beta}} \|\Lambda_{1}^{\beta}\nabla\tilde{b}_{3}\|^\frac{1}{2\beta} \|\partial_{i} b^{\nu}\|^{\frac{1}{2}} \|\partial_{3i} b^{\nu}\|^\frac{1}{2} \\
\leq& C\|(\tilde{u}, \tilde{b})\|_{H^1}^{2}+\frac{1}{18}\|(\Lambda_{1}^{\alpha}\tilde{u}_{1}, \Lambda_{2}^{\alpha}\tilde{u}_{2}, \Lambda_{1}^{\alpha}\tilde{u}_{3})\|_{H^1}^{2}+\frac{1}{18}\|(\Lambda_{2}^{\beta}\tilde{b}_{1}, \Lambda_{1}^{\beta}\tilde{b}_{2}, \Lambda_{2}^{\beta}\tilde{b}_{3})\|_{H^1}^{2}.
\end{align*}
Using the way similar to the term $Q_{41}$, we have that
\begin{align*}
Q_{5}\leq& C\|(\tilde{u}, \tilde{b})\|_{H^1}^{2}+\frac{1}{18}\|(\Lambda_{1}^{\alpha}, \Lambda_{2}^{\alpha}) \tilde{u}\|_{H^1}^{2}+\frac{1}{18}\|(\Lambda_{2}^{\beta}\tilde{b}_{1}, \Lambda_{1}^{\beta}\tilde{b}_{2}, \Lambda_{1}^{\beta}\tilde{b}_{3})\|_{H^1}^{2}, \\
R_{1}\leq& C\|(\tilde{u}, \tilde{b})\|_{H^1}^{2}+\frac{1}{18}\|(\Lambda_{1}^{\alpha}, \Lambda_{2}^{\alpha}) \tilde{u}\|_{H^1}^{2}+\frac{1}{18}\|(\Lambda_{2}^{\beta}\tilde{b}_{1}, \Lambda_{1}^{\beta}\tilde{b}_{2}, \Lambda_{1}^{\beta}\tilde{b}_{3})\|_{H^1}^{2}, \\
R_{21}\leq& C\|(\tilde{u}, \tilde{b})\|_{H^1}^{2}+\frac{1}{18}\|(\Lambda_{1}^{\alpha}\tilde{u}_{1}, \Lambda_{2}^{\alpha}\tilde{u}_{2}, \Lambda_{1}^{\alpha}\tilde{u}_{3})\|_{H^1}^{2}+\frac{1}{18}\|(\Lambda_{2}^{\beta}\tilde{b}_{1}, \Lambda_{1}^{\beta}\tilde{b}_{2}, \Lambda_{2}^{\beta}\tilde{b}_{3})\|_{H^1}^{2}.
\end{align*}
For the term $R_{3}$, by invoking Lemma \ref{lemmage} (iii) and Young inequality, we obtain
\begin{align*}
R_{3}=&-\sum_{i,k=1}^3 \int \partial_{i}u^{\nu}\cdot\nabla\tilde{b}_{k}\cdot\partial_{i}\tilde{b}_{k}dx\\
\leq &C\sum_{i,k=1}^3 \|\partial_{i}\tilde{b}_{k}\|^{1-\frac{1}{2\beta}} \|\Lambda_{k'}^{\beta} \partial_{i}\tilde{b}_{k}\|^\frac{1}{2\beta} \|\nabla\tilde{b}_{k}\|^{1-\frac{1}{2\beta}} \|\Lambda_{k''}^{\beta}\nabla\tilde{b}_{k}\|^\frac{1}{2\beta} \|\partial_{i} u^{\nu}\|^{\frac{1}{2}} \|\partial_{ik} u^{\nu}\|^\frac{1}{2} \\
\leq& C\|\tilde{b}\|_{H^1}^{2}+\frac{1}{18}\sum_{k=1}^3\|(\Lambda_{k'}^{\beta}, \Lambda_{k''}^{\beta})\tilde{b}_{k}\|_{H^1}^{2}.
\end{align*}
To estimate the term $R_{4}$, we decompose it into two terms
\begin{align*}
R_{4}= -\sum_{i=1}^3 \int \partial_{i}\tilde{b}\cdot\nabla u^{0}\cdot\partial_{i}\tilde{b}dx -\sum_{i=1}^3 \int \tilde{b}\cdot\nabla \partial_{i}u^{0}\cdot\partial_{i}\tilde{b}dx:=R_{41}+R_{42}.
\end{align*}
The term $R_{41}$ can be bounded by using again Lemma \ref{lemmage} (iii) and Young inequality
\begin{align*}
R_{41}=&-\sum_{i,k=1}^3 \int \partial_{i}\tilde{b}_{k}\partial_{k} u_{k}^{0}\partial_{i}\tilde{b}_{k}dx-\sum_{i,k=1}^3 \int \partial_{i}\tilde{b}_{k'}\partial_{k'} u_{k}^{0}\partial_{i}\tilde{b}_{k}dx-\sum_{i,k=1}^3 \int \partial_{i}\tilde{b}_{k''}\partial_{k''} u_{k}^{0}\partial_{i}\tilde{b}_{k}dx\\
\leq &C\sum_{i,k=1}^3 \|\partial_{i}\tilde{b}_{k}\|^{2-\frac{1}{\beta}} \|\Lambda_{k'}^{\beta} \partial_{i}\tilde{b}_{k}\|^\frac{1}{2\beta} \|\Lambda_{k''}^{\beta}\partial_{i}\tilde{b}_{k}\|^\frac{1}{2\beta} \|\partial_{k} u^{0}\|^{\frac{1}{2}} \|\partial_{k}^2 u^{0}\|^\frac{1}{2} \\
&+\sum_{i,k=1}^3 \|\partial_{i}\tilde{b}_{k}\|^{1-\frac{1}{2\beta}} \|\Lambda_{k'}^{\beta} \partial_{i}\tilde{b}_{k}\|^\frac{1}{2\beta} \|\partial_{i}\tilde{b}_{k'}\|^{1-\frac{1}{2\beta}} \|\Lambda_{k}^{\beta}\partial_{i}\tilde{b}_{k'}\|^\frac{1}{2\beta} \|\partial_{k'} u^{0}\|^{\frac{1}{2}} \|\partial_{k''k'} u^{0}\|^\frac{1}{2} \\
&+\sum_{i,k=1}^3 \|\partial_{i}\tilde{b}_{k}\|^{1-\frac{1}{2\beta}} \|\Lambda_{k'}^{\beta} \partial_{i}\tilde{b}_{k}\|^\frac{1}{2\beta} \|\partial_{i}\tilde{b}_{k''}\|^{1-\frac{1}{2\beta}} \|\Lambda_{k}^{\beta}\partial_{i}\tilde{b}_{k''}\|^\frac{1}{2\beta} \|\partial_{k''} u^{0}\|^{\frac{1}{2}} \|\partial_{k''}^2 u^{0}\|^\frac{1}{2} \\
\leq& C\|\tilde{b}\|_{H^1}^{2}+\frac{1}{36}\sum_{k=1}^3\|(\Lambda_{k'}^{\beta}, \Lambda_{k''}^{\beta})\tilde{b}_{k}\|_{H^1}^{2}.
\end{align*}
Similarly,
\begin{align*}
R_{42}\leq C\|\tilde{b}\|_{H^1}^{2}+\frac{1}{36}\sum_{k=1}^3\|(\Lambda_{k'}^{\beta}, \Lambda_{k''}^{\beta})\tilde{b}_{k}\|_{H^1}^{2}.
\end{align*}
Consequently, by gathering these estimates above, we obtain that
\begin{align*}
\frac{d}{dt}\|(\tilde{u}, \tilde{b})\|_{H^1}^2 +&\|(\Lambda_{1}^{\alpha}, \Lambda_{2}^{\alpha}) \tilde{u}\|_{H^1}^{2} +\|(\Lambda_2^{\beta}, \Lambda_3^{\beta})\tilde{b}_1\|_{H^1}^2 \\&+\|(\Lambda_1^{\beta}, \Lambda_3^{\beta})\tilde{b}_2\|_{H^1}^2 +\|(\Lambda_1^{\beta}, \Lambda_2^{\beta})\tilde{b}_3\|_{H^1}^2\leq C\|(\tilde{u}, \tilde{b})\|_{H^1}^{2}+C\nu.
\end{align*}
From the Gr\"{o}nwall inequality, we find
\begin{align*}
\|\tilde{u}(\tau)\|_{H^1}^2+\|\tilde{b}(\tau)\|_{H^1}^2\leq Ct\nu.
\end{align*}
This ends the proof of Theorem \ref{theorem2}.
\end{proof}

\noindent\textbf{Acknowledgments.} The work of Aibin Zang was supported in part  by the National Natural Science Foundation of China (Grant no. 12261039, 12061080), and Jiangxi Provincial Natural Science Foundation  (No. 20224ACB201004). The work of Yuelong Xiao was supported by National Natural Science Foundation of China (Grant no.11871412). The research of Xuemin Deng partially Supported by Postgraduate Scientific Research Innovation Project of Hunan Province (Grant no.CX20210603, XDCX2021B096).


\end{document}